\documentclass{amsart}
\usepackage{amssymb}
\usepackage{hyperref}
\usepackage{url}
\usepackage[all]{xy}

\newtheorem{thm}{Theorem}
\newtheorem{introthm}{Theorem}
\newtheorem{lem}[thm]{Lemma}
\newtheorem{cor}[thm]{Corollary}
\newtheorem{prop}[thm]{Proposition}
\newtheorem{conj}[thm]{Conjecture}

\newtheorem{exa}[thm]{Example}

\theoremstyle{definition}
\newtheorem{defn}[thm]{Definition}
\newtheorem{rem}[thm]{Remark}

\numberwithin{thm}{section}
\numberwithin{equation}{section}

\newfont{\cyrr}{wncyr10}
\def\Sh{\mbox{\cyrr Sh}}

\def\Z{\mathbf{Z}}
\def\Q{\mathbf{Q}}
\def\F{\mathbf{F}}
\def\R{\mathbf{R}}
\def\C{\mathbf{C}}
\def\bA{\mathbf{A}}

\def\G{\mathbf{G}}

\def\Qp{\Q_p}
\def\Fp{\F_p}

\def\cS{\mathcal{S}}
\def\A{\mathcal{A}}

\def\O{\mathcal{O}}

\def\cP{\mathcal{P}}

\def\I{\mathcal{I}}
\def\B{\mathcal{B}}

\def\ld{\mathcal{h}}
\def\rd{\mathcal{i}}

\def\l{\mathfrak{q}}

\def\p{\mathfrak{p}}

\def\D{\mathcal{D}}
\def\d{\mathfrak{d}}

\def\Hom{\mathrm{Hom}}
\def\Gal{\mathrm{Gal}}

\def\rk{\mathrm{rank}}
\def\ord{\mathrm{ord}}

\def\Aut{\mathrm{Aut}}
\def\Sel{\mathrm{Sel}}
\def\End{\mathrm{End}}
\def\Res{\mathrm{Res}}
\def\Frob{\mathrm{Frob}}

\def\ur{\mathrm{ur}}
\def\GL{\mathrm{GL}}
\def\SL{\mathrm{SL}}

\def\ST{\Sh\mathrm{T}}
\def\loc{\mathrm{loc}}

\def\image{\mathrm{image}}
\def\Pic{\mathrm{Pic}}
\def\ram{\mathrm{ram}}

\def\Isom{\mathrm{Isom}}

\def\parity{\rho}

\def\ord{\mathrm{ord}}

\def\N{\mathbf{N}}

\def\Kb{\bar{K}}

\def\too{\longrightarrow}
\def\map#1{\;\xrightarrow{#1}\;}
\def\isom{\xrightarrow{\sim}}
\def\hookto{\hookrightarrow}
\def\onto{\twoheadrightarrow}
\def\dirsum#1{\underset{#1}{\textstyle\bigoplus}}

\def\Hu{H^1_{\ur}}

\def\HF{H^1_{\cS}}

\def\bmu{\boldsymbol{\mu}}
\def\bq{\mathbf{q}}
\def\Fset{\mathcal{F}}
\def\Xset{\mathcal{C}}
\def\Xsetclassic{\Xset'}
\def\H{\mathcal{H}}

\def\sgn{\mathrm{sign}_\Delta}

\def\mv#1{m_v^{#1}}

\def\one{\mathbf{1}}
\def\iK{\bA_K^\times}
\def\com#1#2{h_v(#1,#2)}
\def\compsi#1#2{h_v^\psi(#1,#2)}

\def\mag#1{|| #1 ||}
\def\e{\varepsilon}
\def\metabolic{metabolic }
\def\Metabolic{Metabolic }
\def\dm#1{\| #1 \|}

\title[Disparity in Selmer ranks of quadratic twists of elliptic curves]{Disparity in Selmer 
   ranks of quadratic twists of elliptic curves}

\author{Zev Klagsbrun}
\address{Department of Mathematics, 
University of Wisconsin\,-\,Madison,
Madison, WI 53706, 
USA}
\email{\href{mailto:klagsbru@math.wisc.edu}{klagsbru@math.wisc.edu}}
\author{Barry Mazur}
\address{Department of Mathematics, 
Harvard University,
Cambridge, MA 02138, 
USA}
\email{\href{mailto:mazur@math.harvard.edu}{mazur@math.harvard.edu}}
\author{Karl Rubin}
\address{Department of Mathematics, 
UC Irvine,
Irvine, CA 92697, 
USA}
\email{\href{mailto:krubin@math.uci.edu}{krubin@math.uci.edu}}

\subjclass[2010]{Primary: 11G05, Secondary: 11G40}

\thanks{This material is based upon work supported by the 
National Science Foundation under grants DMS-0700580, DMS-0757807, and DMS-0968831.  
Much of this work was carried out while the second and third authors were in residence at 
MSRI, and they would also like to thank MSRI for support and hospitality.
We would also like to thank the referee for helpful comments.}

\begin{document}

\begin{abstract}
We study the parity of $2$-Selmer ranks in the family of quadratic twists of 
an arbitrary elliptic curve $E$ over an arbitrary number field $K$. We prove 
that the fraction of  twists (of a given elliptic curve over a fixed number field) 
having even $2$-Selmer rank exists as a stable limit over the family of twists, and 
we compute this fraction as an explicit product of local factors. We give an example 
of an elliptic curve $E$ such that as $K$ varies, these fractions are dense in $[0, 1]$.
More generally, our results also apply to $p$-Selmer ranks of twists of $2$-dimensional 
self-dual ${\bf F}_p$-representations of the absolute Galois group of $K$ by characters 
of order $p$.
\end{abstract}

\maketitle

\tableofcontents

\section*{Introduction}

The type of question that we consider in this paper has its roots in a conjecture of Goldfeld 
\cite[Conjecture B]{goldfeld} on the distribution of Mordell-Weil ranks in the family of 
quadratic twists of an arbitrary elliptic curve over $\Q$, and a result of Heath-Brown 
\cite[Theorem 2]{heath-brown} on the distribution of $2$-Selmer ranks in the family of 
quadratic twists over $\Q$ of the  elliptic curve $y^2 = x^3 - x$.

We study here the distribution of the parities of $2$-Selmer ranks in the family of 
quadratic twists of an arbitrary elliptic curve $E$ over an arbitrary number field $K$. 
For example, let $\parity(E/K)$ be the fraction of quadratic twists of $E/K$ that have 
odd $2$-Selmer rank.  Precisely, for real numbers $X > 0$ let
$$
\Xset(K,X) := \bigl\{\chi : G_K \to \{\pm1\} : \text{$\chi$ is ramified only at primes $\l$ with $\N\l \le X$}\bigr\}
$$
and define
$$
\parity(E/K) := \lim_{X \to \infty}\frac{|\{\chi\in\Xset(K,X) : 
   \text{$\dim_{\F_2}\Sel_2(E^\chi/K)$ is odd}\}|}{|\Xset(K,X)|}.
$$
It follows from a result of Monsky \cite[Theorem 1.5]{monsky} along with root number 
calculations that $\parity(E/\Q) = 1/2$ for every elliptic curve $E/\Q$.  
It had already been noticed (see \cite{dokchitser}) that this is not true when $\Q$ 
is replaced by an arbitrary number field $K$, because there are examples with $K \ne \Q$ 
for which $\parity(E/K) = 0$, and others with $\parity(E/K) = 1$.  
Our main theorem (see Theorem \ref{sllem,p=2}) evaluates $\parity(E/K)$.

\begin{introthm}
\label{thma}
Suppose $E$ is an elliptic curve defined over a number field $K$.  Then for all sufficiently 
large $X$ we have
$$
\parity(E/K) = \frac{|\{\chi\in\Xset(K,X) : 
   \text{$\dim_{\F_2}\Sel_2(E^\chi/K)$ is odd}\}|}{|\Xset(K,X)|} = (1-\delta(E/K))/2
$$
where $\delta(E/K) \in [-1,1] \cap \Z[1/2]$ is given by an explicit finite product of 
local factors (see Definition \ref{epdef}).
\end{introthm}

We call $\delta(E/K)$ the ``disparity'' in the distribution of $2$-Selmer ranks of twists of $E$.
If $K$ has a real embedding then $\delta(E/K) = 0$ so $\parity(E/K) = 1/2$ (see Corollary \ref{6.8}).  
On the other hand, 
Example \ref{50a1} exhibits a particular elliptic curve $E/\Q$ such that as $K$ varies, the 
set $\{\delta(E/K)\}$ is dense in $[-1,1]$, so $\{(1-\delta(E/K))/2\}$ is dense in $[0,1]$.  

The finiteness of the $2$-part of the Shafarevich-Tate group would imply that the parity 
of the $2$-Selmer rank is the same as the parity of the Mordell-Weil rank.  Thus one would 
expect that Theorem \ref{thma} holds with $2$-Selmer rank replaced by Mordell-Weil rank.  
Further, Theorem \ref{thma} suggests a natural generalization of Goldfeld's conjecture 
(see Conjecture \ref{gcK}).

In a forthcoming paper \cite{kmr2}, we will use the methods of this paper to make a finer study of the 
distribution of $2$-Selmer ranks, inspired by the work of Heath-Brown \cite{heath-brown}, 
Swinnerton-Dyer \cite{sw-d}, and Kane \cite{kane}.

Our methods begin with those of \cite{MRvisibility} and \cite{MRstablerank}.
Namely, we view all of the Selmer groups $\Sel_2(E^\chi/K)$ as subspaces of $H^1(K,E[2])$, 
defined by local conditions that vary with $\chi$.  In this way we can attach a Selmer group to 
a collection of local quadratic characters.  The question of which collections of 
local characters arise from global characters is an exercise in class field theory 
(see \S\ref{lgc}).

To prove Theorem \ref{thma}, we show that the parity of 
$\dim_{\F_2}\Sel_2(E^\chi/K)$ depends only on the restrictions of $\chi$ to the 
decomposition groups at places dividing $2 \Delta_E \infty$, where $\Delta_E$ is 
the discriminant of some model of $E$ (see Proposition \ref{5.1,p=2}).  
(This is consistent with the behavior of the global root numbers of twists of $E$.)
In particular the map that sends a character $\chi \in \Hom(G_K,\{\pm 1\})$ to 
the parity of $\dim_{\F_2}\Sel_2(E^\chi/K)$ factors through the finite quotient 
$\prod_{v \mid 2\Delta_E\infty}\Hom(G_{K_v},\{\pm 1\})$.
Using this fact we are able to deduce Theorem \ref{thma}.

There is another important ingredient in the proof of Theorem \ref{thma}.  
We make essential use of a recent observation of Poonen and Rains \cite{poonenrains} that 
the local conditions that define the $2$-Selmer groups we are studying are maximal isotropic subspaces 
for a natural quadratic form on the local cohomology groups $H^1(K_v,E[2])$.  We use this in a 
crucial way in the proof of Theorem \ref{compsel}, which extends a result from \cite{MRvisibility} 
to include the case $p=2$.  Theorem \ref{compsel} is a key ingredient in the proof of  
Theorem \ref{thma}.

Our methods apply much more generally than to $2$-Selmer groups of elliptic curves, and 
throughout this paper we work in this fuller generality.  Namely, suppose $p$ is any prime, 
and $T$ is a $2$-dimensional $\Fp$-vector space with 
\begin{itemize}
\item
an action of the absolute Galois group $G_K$, 
\item
a nondegenerate $G_K$-equivariant alternating $\bmu_p$-valued pairing, and 
\item
a ``global \metabolic structure'' (see Definition \ref{ahsdef}).  
\end{itemize}
We also assume we are given ``twisting data'' 
(Definition \ref{twistdata}) that allows us to define a family of Selmer groups 
$\Sel(T,\chi)$ as $\chi$ runs through characters of $G_K$ of order $p$.  We have analogues 
of Theorem \ref{thma} describing the distribution of $\dim_{\Fp}\Sel(T,\chi)$ 
in this setting.

For example, if $E$ is an elliptic curve over $K$, then $T := E[p]$, the kernel of 
multiplication by $p$ on $E$, comes equipped with 
all the structure we require.  When $p > 2$ the Selmer groups $\Sel(E[p],\chi)$ 
are not Selmer groups of elliptic curves, but they are Selmer groups of $(p-1)$-dimensional 
abelian varieties over $K$ that are twists of $E$ in the sense of \cite{mrs}.  See \S\ref{eex}, 
and see Theorem \ref{sllem} for the analogue of Theorem \ref{thma} in this setting.

The layout of the paper is as follows.  Let $T$ be a Galois module as above.  
In \S\ref{MS} we derive some elementary properties of Lagrangian subspaces in quadratic 
vector spaces that we will need in the sequel.
In \S\ref{hsss} we define \metabolic structures 
and Selmer groups in the generality we will need them.  The key result is Theorem \ref{compsel}, 
which shows how the parity of the Selmer rank changes when we change some of the defining 
local conditions.  In \S\ref{twistsec} we define the Selmer groups associated to twists of $T$, 
and in \S\ref{eex} we show how these Selmer groups reduce to classical Selmer groups of twists 
when $T = E[p]$ for an elliptic curve $E/K$.

Section \ref{lgc} uses class field theory to allow us to go back and 
forth between global characters and collections of local characters.  In \S\ref{pd2} we 
study ``parity disparity'' when $p=2$, and prove Theorem \ref{thma}, and in \S\ref{pdp} 
we obtain similar but not identical results when $p>2$.

\section{Notation}
\label{notation}
Fix a number field $K$ and a rational prime $p$.  Let $\Kb$ denote a fixed algebraic 
closure of $K$, and $G_K := \Gal(\Kb/K)$.  Let $\bmu_p$ denote the group of $p$-th roots 
of unity in $\Kb$.

Throughout this paper $T$ will denote a two-dimensional 
$\Fp$-vector space with a continuous action of $G_K$, and with a nondegenerate 
$G_K$-equivariant alternating pairing corresponding to an isomorphism
\begin{equation}
\label{weilpair}
\wedge^2T \isom \bmu_p.
\end{equation}

We will use $v$ (resp., $\l$) for a place (resp., nonarchimedean place, or prime ideal) of $K$.
If $v$ is a place of $K$, we let $K_v$ denote the completion of $K$ at $v$, and 
$K_v^\ur$ its maximal unramified extension.  
We say that $T$ is unramified at $v$ if the inertia subgroup of $G_{K_v}$ 
acts trivially on $T$, and in that case we define the unramified subgroup 
$\Hu(K_v,T) \subset H^1(K_v,T)$ by
$$
\Hu(K_v,T) := H^1(K_v^\ur/K_v,T) = \ker [H^1(K_v,T) \to H^1(K_v^\ur,T)].
$$
If $c \in H^1(K,T)$ and $v$ is a place of $K$, we will often abbreviate 
$c_v := \loc_v(c)$ for the localization of $c$ in $H^1(K_v,T)$.

We also fix a finite set $\Sigma$ of places of $K$, containing all places where $T$ is ramified, all 
primes above $p$, and all archimedean places.

\section{\Metabolic spaces}
\label{MS}

Fix for this section a finite dimensional $\Fp$-vector space $V$.

\begin{defn}
\label{qfdef}
A {\em quadratic form} on $V$ is a function $q : V \to \Fp$ such that 
\begin{itemize}
\item
$q(av) = a^2 q(v)$ for every $a \in \Fp$ and $v \in V$,
\item
the map $(v,w)_q := q(v+w)-q(v)-q(w)$ is a bilinear form.
\end{itemize}
If $X \subset V$, we denote by $X^\perp$ the orthogonal complement of $X$ in $V$ under 
the pairing $(\;\;,\,\;)_q$.  We say that $(V,q)$ is a {\em \metabolic space} if 
$(\;\;,\,\;)_q$ is nondegenerate and $V$ has a subspace 
$X$ such that $X = X^\perp$ and $q(X) = 0$.  Such a subspace $X$ is called 
a {\em Lagrangian subspace} of $V$.
\end{defn}

For this section, if $W$ is an $\Fp$-vector space we let $\dm{W} := \dim_{\Fp}(W)$.

\begin{lem}
\label{ell0}
Suppose $(V,q)$ is a \metabolic space, $X$ is a Lagrangian subspace, and $W$ is a subspace 
of $V$ such that $q(W) = 0$.  
Then $W^\perp \cap X + W$ is a Lagrangian subspace of $V$.
\end{lem}

\begin{proof}
Exercise.  See for example \cite[Remark 2.4]{poonenrains}.
\end{proof}

\begin{lem}
\label{ell1}
Suppose $(V,q)$ is a \metabolic space and $X$, $Y$, and $Z$ are Lagrangian subspaces of $V$.  
Then
$$
\dm{(X+Y) \cap Z} \equiv \dm{(X \cap Z) + (Y \cap Z)} \pmod{2}.
$$
\end{lem}

\begin{proof}
We adapt the proof of \cite[Proposition 1.3]{MRvisibility} (see also \cite[Lemma 1.5.7]{howard}).  
We define an alternating nondegenerate 
pairing on $((X+Y) \cap Z)/(X \cap Z + Y \cap Z)$, as follows.

Suppose $z, z' \in (X+Y) \cap Z$.  Write $z = x+y$ and $z' = x'+y'$ with $x,x' \in X$ and $y \in Y$, and define 
\begin{equation}
\label{hpdef}
[z,z'] := (x,y')_q.
\end{equation}
Note that $x$ and $y'$ are well-defined modulo $X \cap Y = (X+Y)^\perp$, so $[z,z']$ does 
not depend on the choice of $x$ or $y'$.  Thus $[\;\;,\;\,]$ is a well-defined bilinear pairing 
on $(X+Y) \cap Z$.  

By definition we have 
$$
[z,z] = (x,y)_q = q(x+y) - q(x) - q(y).
$$
Since $X$, $Y$, and $Z$ are Lagrangian, and $x+y = z$, we have $q(x+y) = q(x) = q(y) = 0$, 
so $[z,z] = 0$ for every $z$, i.e., $[\;\;,\;\,]$ is alternating (and therefore also skew-symmetric).

If $z \in Y \cap Z$ then we can take $x = 0$ in \eqref{hpdef}, so $[z,z'] = 0$ for every $z'$.  
Using the skew-symmetry we deduce that $Y \cap Z$ is in the (left and right) 
kernel of the pairing    
$[\;\;,\;\,]$.  Similarly $X \cap Z$, and hence $X \cap Z + Y \cap Z$, is in the kernel.

Conversely, if $z$ is in the kernel of this pairing, then 
(still writing $z = x+y$ with $x \in X$ and $y \in Y$)
$$
0 = [z,z'] = (x,y')_q = (x,z')_q  = (x,z'+x'')_q
$$
for every $z' \in (X+Y) \cap Z$ and $x'' \in X \cap Y$.  
Applying Lemma \ref{ell0} with $W = X\cap Y$, $W^\perp = X+Y$ we see that 
$$
x \in ((X+Y) \cap Z + X \cap Y)^\perp = (X+Y) \cap Z + X \cap Y.
$$
Thus, modifying $x$ and $y$ by an element of $X \cap Y$, we may assume that $x \in Z$, 
and then $y = z-x \in Z$ as well, so $z \in X \cap Z + Y \cap Z$.  

This completes the proof that the pairing \eqref{hpdef} is alternating and nondegenerate 
on $((X+Y) \cap Z)/(X \cap Z + Y \cap Z)$.  A standard argument now shows that the dimension
$\dm{((X+Y) \cap Z)/(X \cap Z + Y \cap Z)}$ is even, and the lemma follows.
\end{proof}

\begin{prop}
\label{ell2}
Suppose $(V,q)$ is a \metabolic space and $X$, $Y$, and $Z$ are Lagrangian subspaces of $V$.  
Then
$$
\textstyle
\dm{X \cap Y} + \dm{Y \cap Z} + \dm{X \cap Z} \equiv \frac{1}{2}\dm{V} = \dm{X} = \dm{Y} = \dm{Z} \pmod{2}.
$$
\end{prop}

\begin{proof}
For subspaces $U, W$ of $V$ we have $\dm{U}+\dm{W} = \dm{U+W} + \dm{U \cap W}$.  
This identity gives the two equalities below, and Lemma \ref{ell1} gives the congruence:
\begin{align*}
\dm{X \cap Y} + \dm{Y \cap Z} + \dm{&X \cap Z} 
   = \dm{X \cap Y} + \dm{X \cap Z + Y \cap Z} + \dm{X \cap Y \cap Z} \\
   &\equiv \dm{X \cap Y} + \dm{(X+Y) \cap Z} + \dm{X \cap Y \cap Z} \pmod{2}\\
   &= \dm{(X+Y) \cap Z + X \cap Y} + 2\dm{X \cap Y \cap Z}.
\end{align*}
By Lemma \ref{ell0}, the subspace $(X+Y) \cap Z + X \cap Y$ is Lagrangian, and all 
Lagrangian subspaces have the same dimension $\frac{1}{2}\dm{V}$, so this completes the proof.
\end{proof}

\begin{cor}
\label{ell3}
Suppose $(V,q)$ is a \metabolic space and $X$, $Y$, and $Z$ are Lagrangian subspaces of $V$.  
Then
$$
\dm{X/(X \cap Y)} + \dm{Y/(Y \cap Z)} + \dm{Z/(X \cap Z)}  \equiv 0 \pmod{2}. 
$$
\end{cor}

\begin{proof}
This follows directly from Proposition \ref{ell2}.
\end{proof}

\section{\Metabolic structures and Selmer structures}
\label{hsss}

In this section we define what we mean by a global \metabolic structure $\bq$ on $T$, 
and by a Selmer group for $T$ and $\bq$.  The main result is Theorem \ref{compsel}, 
which shows how the parity of the Selmer rank changes when we change the defining local 
conditions.

The cup product and the pairing \eqref{weilpair} induce a pairing
$$
H^1(K,T) \times H^1(K,T) \map{\;\cup\;} H^2(K,T \otimes T) \too H^2(K,\bmu_p).
$$
If $v$ is a place of $K$ and $K_v$ is the completion of $K$ at $v$, then 
applying the same construction over the field $K_v$ gives local pairings
$$
H^1(K_v,T) \times H^1(K_v,T) \map{\;\cup\;} H^2(K_v,T \otimes T) \too H^2(K_v,\bmu_p).
$$
For every $v$ there is a canonical inclusion $H^2(K_v,\bmu_p) \hookto \Fp$
that is an isomorphism unless either $K_v = \C$, or $K_v = \R$ and $p > 2$.
The local Tate pairing is the composition
\begin{equation}
\label{tatepair}
\ld \;\;,\;\rd_v : H^1(K_v,T) \times H^1(K_v,T) \too \Fp.
\end{equation}

The Tate pairings satisfy the following well-known properties.

\begin{thm}
\label{tateprop}
\hfill
\begin{enumerate}
\item
For every $v$, the pairing $\ld \;\;,\;\rd_v$ is symmetric and nondegenerate.
\item
If $v \notin \Sigma$ then $\Hu(K_v,T) \subset H^1(K_v,T)$ is equal to its own orthogonal complement 
under $\ld \;\;,\;\rd_v$.
\item
If $c, d \in H^1(K,T)$, then $\ld c_v,d_v \rd_v = 0$ for almost all $v$ and 
$\sum_v \ld c_v,d_v \rd_v = 0$.
\end{enumerate} 
\end{thm}

\begin{proof}
For (i) and (ii), see for example \cite[Corollary  I.2.3 and Theorem I.2.6]{milne}.
The first part of (iii) follows from (ii), and the second from the fact that 
the sum of the local invariants of an element of the global Brauer group is zero.
\end{proof}

\begin{defn}
Suppose $v$ is a place of $K$.  
We say that $q$ is a {\em Tate quadratic form} on $H^1(K_v,T)$ if the bilinear form 
induced by $q$ (Definition \ref{qfdef}) is $\ld \;\;,\;\, \rd_v$.  
If $v \notin \Sigma$, then 
we say that $q$ is {\em unramified} if $q(x) = 0$ for all $x \in \Hu(K_v,T)$.
\end{defn}

\begin{defn}
\label{ahsdef}
Suppose $T$ is as above.  A {\em global \metabolic structure} $\bq$ on $T$ consists 
of a Tate quadratic form $q_v$ on $H^1(K_v,T)$ for every place $v$, such that
\begin{enumerate}
\item
$(H^1(K_v,T),q_v)$ is a \metabolic space for every $v$,
\item
if $v \notin \Sigma$ then $q_v$ is unramified,
\item
if $c \in H^1(K,T)$ then $\sum_v q_v(c_v) = 0$.
\end{enumerate} 
\end{defn}

Note that if $c \in H^1(K,T)$ then $c_v \in \Hu(K_v,T)$ for almost all $v$, 
so the sum in Definition \ref{ahsdef}(iii) is finite.

\begin{lem}
\label{2.4}
If $p > 2$ then there is a unique Tate quadratic form $q_v$ on $H^1(K_v,T)$ for every $v$, 
and a unique global \metabolic structure on $T$.
\end{lem}

\begin{proof}
Since $p \ne 2$, for every $v$ 
there is a unique Tate quadratic form $q_v$ on $H^1(K_v,T)$, namely 
$$
q_v(x) := \textstyle\frac{1}{2}\ld x,x\rd_v.
$$
If $v \notin \Sigma$ then $q_v$ is unramified by Theorem \ref{tateprop}(ii), 
and if $c \in H^1(K,T)$ then 
$\sum_v q_v(c_v) = \frac{1}{2}\sum_v \ld c_v,c_v \rd_v = 0$ by 
Theorem \ref{tateprop}(iii).
\end{proof}

\begin{rem}
Suppose $p = 2$ and $\bq = \{q_v : \text{$v$ of $K$}\}$ is a global \metabolic structure on $T$. 
If $c \in H^1(K,T)$ is such that $c_v \in \Hu(K_v,T)$ for every $v \notin \Sigma$,  
then for every $v$ we can define a new Tate quadratic form $q'_v$ on $H^1(K_v,T)$ by 
$$
q'_v(x) := q_v(x) + \ld x,c_v \rd_v.
$$
It is straightforward to check (using Theorem \ref{tateprop}) that $\bq' := \{q'_v\}$ is again 
a global \metabolic structure on $T$, and if $c \ne 0$ then $\bq' \ne \bq$.
\end{rem}

\begin{defn}
Suppose $v$ is a place of $K$ and $q_v$ is a quadratic form on $H^1(K_v,T)$.  
Let 
$$
\H(q_v) := \{\text{Lagrangian subspaces of $(H^1(K_v,T),q_v)$}\},
$$
and if $v \notin \Sigma$
$$
\H_\ram(q_v) := \{X \in \H(q_v) : X \cap \Hu(K_v,T) = 0\}.
$$
\end{defn}

\begin{lem}
\label{countlem}
Suppose $v \notin \Sigma$ and $q_v$ is a Tate quadratic form on $H^1(K_v,T)$.  
Let $d_v := \dim_{\Fp}T^{G_{K_v}}$.  
Then:
\begin{enumerate}
\item
$\dim_{\Fp}H^1(K_v,T) = 2d_v$, 
\item
every $X \in \H(q_v)$ has dimension $d_v$, 
\item
if $d_v > 0$ and $q_v$ is unramified, then 
$|\H_\ram(q_v)| = p^{d_v-1}$.
\end{enumerate}
\end{lem}

\begin{proof}
Since $\ld\;,\;\rd_v$ is nondegenerate, every Lagrangian 
subspace of $H^1(K_v,T)$ has dimension $\frac{1}{2}\dim_{\Fp}H^1(K_v,T)$.
Since $v \notin \Sigma$, Theorem \ref{tateprop}(ii) shows that 
$\Hu(K_v,T)$ is Lagrangian.  
We have 
$
\Hu(K_v,T) = T/(\Frob_v-1)T
$
(see for example \cite[\S XIII.1]{serrecg}), so the exact sequence 
$$0 \too T^{G_{K_v}} \too T \map{\Frob_v-1} T \too T/(\Frob_v-1)T \too 0$$
shows that $\dim_{\Fp}\Hu(K_v,T) = d_v$.  
This proves (i) and (ii).

Assertion (iii) follows from a calculation of Poonen and Rains 
\cite[Proposition 2.6(b,e)]{poonenrains}.
\end{proof}

\begin{defn}
\label{ssdef}
Suppose $T$ is as above and $\bq$ is a global \metabolic structure on $T$.
A {\em Selmer structure} $\cS$ for $(T,\bq)$ (or simply for $T$, if $\bq$ is understood) 
consists of 
\begin{itemize}
\item
a finite set $\Sigma_\cS$ of places of $K$, containing $\Sigma$,
\item
for every $v \in \Sigma_\cS$, a Lagrangian subspace $\HF(K_v,T) \subset H^1(K_v,T)$.
\end{itemize}
If $\cS$ is a Selmer structure, we set $\HF(K_v,T) := \Hu(K_v,T)$ if $v \notin \Sigma_\cS$, 
and we define the {\em Selmer group} $\HF(K,T) \subset H^1(K,T)$ 
by 
$$
\HF(K,T) := \ker (H^1(K,T) \too \dirsum{v}H^1(K_v,T)/\HF(K_v,T)),
$$
i.e., the subgroup of $c \in H^1(K,T)$ such that $c_v \in \HF(K_v,T)$ for every $v$.
\end{defn}

\begin{thm}
\label{compsel}
Suppose $\cS$ and $\cS'$ are two Selmer structures for $T$.  Then
\begin{multline*}
\dim_{\Fp}\HF(K,T) - \dim_{\Fp}H^1_{\cS'}(K,T)  \\
   \equiv \sum_{v \in \Sigma_\cS \cup \Sigma_{\cS'}} 
      \dim_{\Fp}\HF(K_v,T)/(\HF(K_v,T)\cap H^1_{\cS'}(K_v,T)) \pmod{2}.
\end{multline*}
\end{thm}

\begin{proof}
When $p > 2$, this is \cite[Theorem 1.4]{MRvisibility}.  
We will prove this for all $p$ using Proposition \ref{ell2}.

Let $\Sigma' := \Sigma_\cS \cup \Sigma_{\cS'}$. 
Define $V = \prod_{v \in \Sigma'}H^1(K_v,T)$, 
so $(V,\sum_{v}q_v)$ is a \metabolic space.  
Let $\loc_{\Sigma'} : H^1(K,T) \to V$ denote the product of the localization maps.
Define three subspaces of $V$
\begin{itemize}
\item
$X := \prod_{v \in \Sigma'} \HF(K_v,T)$,
\item
$Y := \prod_{v \in \Sigma'} H^1_{\cS'}(K_v,T)$,
\item
$Z$ is the image under $\loc_{\Sigma'}$ of 
$\ker (H^1(K,T) \to \hskip -2pt\dirsum{v \notin \Sigma'}H^1(K_v,T)/H^1_{\ur}(K_v,T))$.
\end{itemize}
The spaces $X$ and $Y$ are Lagrangian by definition of Selmer structure.  
That $Z$ is also Lagrangian can be seen as follows.  We have $Z^\perp = Z$ by Poitou-Tate 
global duality (see for example \cite[Theorem I.4.10]{milne}, \cite[Theorem 3.1]{tate}, 
or \cite[Theorem 1.7.3]{rubinES}).  If $z \in Z$, then $z = \loc_{\Sigma'}(s)$ with 
$s \in H^1(K,T)$ satisfying $s_v \in H^1_\ur(K_v,T)$ for every $v \notin \Sigma'$.  
Then $q_v(s_v) = 0$ if $v \notin \Sigma'$ by Definition \ref{ahsdef}(ii), so
$$
\bigl(\sum_{v \in \Sigma'}q_v\bigr)(z) = \sum_{v \in \Sigma'}q_v(s_v) = \sum_{\text{all $v$}}q_v(s_v) = 0
$$
by Definition \ref{ahsdef}(iii).  Thus $Z$ is Lagrangian.

Note that from the definitions we have exact sequences
$$
\xymatrix@R=10pt{
0 \ar[r] & A \ar[r] & \HF(K,T) \ar^-{\loc_{\Sigma'}}[r] & X \cap Z \ar[r] & 0\\
0 \ar[r] & A \ar[r] & H^1_{\cS'}(K,T) \ar^-{\loc_{\Sigma'}}[r] & Y \cap Z \ar[r] & 0\\
}
$$
where the kernel $A$ in both sequences is 
$$
A = \ker \bigl(H^1(K,T) \too \dirsum{v \notin \Sigma'}H^1(K_v,T)/H^1_{\ur}(K_v,T)  \dirsum{v \in \Sigma'}H^1(K_v,T)\bigr)
$$
Thus by Proposition \ref{ell2} we have
\begin{align*}
\dim_{\Fp}\HF(K,T) - \dim_{\Fp}H^1_{\cS'}(K,T) 
   &= \dim_{\Fp}(Y \cap Z) - \dim_{\Fp}(X \cap Z) \\
   &\equiv \dim_{\Fp}(X/(X\cap Y)) \pmod{2}.
\end{align*}
Since $X/(X\cap Y) = \prod_{v \in \Sigma'}\HF(K_v,T)/(\HF(K_v,T)\cap H^1_{\cS'}(K_v,T))$, this 
completes the proof of the theorem.
\end{proof}

\section{Twisted Selmer groups}
\label{twistsec}

Given $T$ as above (and some additional ``twisting data'', see Definition \ref{twistdata}), 
in this section we show 
how to attach to every character $\chi \in \Hom(G_K,\bmu_p)$ a Selmer group $\Sel(T,\chi)$.
More generally, we attach a Selmer group $\Sel(T,\gamma)$ to every collection of local characters 
$\gamma = (\gamma_v)$ with $\gamma_v \in \Hom(G_{K_v}, \bmu_p)$ for $v$ in some finite set 
containing $\Sigma$.  Our main result is Theorem \ref{kramer}, which uses Theorem \ref{compsel} 
to show how the parity of the Selmer rank changes when we change some of the $\gamma_v$.

\begin{defn}
\label{setdef}
If $L$ is a field, define
$$
\Xset(L) := \Hom(G_L,\bmu_p)
$$
(throughout this paper, ``$\Hom$'' will always mean continuous homomorphisms).
If $L$ is a local field, we let $\Xset_\ram(L) \subset \Xset(L)$ denote the 
subset of ramified characters.  In this case local class field theory identifies 
$\Xset(L)$ with $\Hom(L^\times,\bmu_p)$, and $\Xset_\ram(L)$ is then the subset of 
characters nontrivial on the local units $\O_L^\times$.
Let $\one_L \in \Xset(L)$ denote the trivial character.

There is a natural action of $\Aut(\bmu_p) = \Fp^\times$ on $\Xset(L)$, and we let 
$\Fset(L) := \Xset(L)/\Aut(\bmu_p)$.  Then $\Fset(L)$ is naturally identified with 
the set of cyclic extensions of $L$ of degree dividing $p$, via the correspondence 
that sends $\chi \in \Xset(L)$ to the fixed field $\bar{L}^{\ker(\chi)}$ of $\ker(\chi)$ 
in $\bar{L}$.  If $L$ is a local field, then $\Fset_\ram(L)$ denotes the set of ramified 
extensions in $\Fset(L)$.
\end{defn}
  
\begin{defn}
For $1 \le i \le 2$ define
$$
\cP_i := \{\l : \text{$\l \notin \Sigma$, $\bmu_p \subset K_\l$, 
   and $\dim_{\Fp}T^{G_{K_\l}} = i$}\},
$$
and $\cP_0 := \{\l : \l \notin \Sigma \cup \cP_1 \cup \cP_2\}$.  
Define the {\em width} $w(\l) \in \{0,1,2\}$ of a prime $\l$ of $K$, $\l \notin \Sigma$, by
$w(\l) := i$ if $\l\in\cP_i$.
\end{defn}

Let $K(T)$ denote the field of definition of the elements of $T$, i.e., 
the fixed field in $\Kb$ of $\ker (G_K \to \Aut(T))$.

\begin{lem}
\label{4.2}
Suppose $\l$ is a prime of $K$, $\l \notin \Sigma$, and let $\Frob_\l \in \Gal(K(T)/K)$ 
be a Frobenius element for some choice of prime above $\l$.  Then
\begin{enumerate}
\item
$\l\in\cP_2$ if and only if $\Frob_\l = 1$,
\item
$\l\in\cP_1$ if and only if $\Frob_\l$ has order exactly $p$,
\item
$\l\in\cP_0$ if and only if $\Frob_\l^p \ne 1$.
\end{enumerate}
In particular $\cP_2$ has positive density in the set of all primes of $K$, 
and $\cP_1$ has positive density if and only if $p \mid [K(T):K]$.
\end{lem}

\begin{proof}
Fix an $\Fp$-basis of $T$ so that we can view $\Frob_\l \in \GL_2(\Fp)$.  
Then by \eqref{weilpair}
$$
\bmu_p \subset K_\l \iff \text{$\Frob_\l$ acts trivially on $\bmu_p$} 
   \iff \det(\Frob_\l) = 1.
$$
Since $\l \notin \Sigma$, $T$ is unramified at $\l$, so $T^{G_{K_\l}} = T^{\Frob_\l = 1}$, 
the subspace of $T$ fixed by $\Frob_\l$. 
We have $\dim_{\Fp}T^{\Frob_\l = 1} = 2$ if and only if 
$\Frob_\l = 1$, and if $\det(\Frob_\l) = 1$, then 
$\dim_{\Fp}T^{\Frob_\l = 1} = 1$ if and only if $\Frob_\l$ has order $p$.
This proves the lemma.
\end{proof}

\begin{defn}
\label{twistdata}
Suppose $T$, $\Sigma$ are as above, and $\bq$ is a global \metabolic structure on $T$.
By {\em twisting data} we mean
\begin{enumerate}
\item
for every $v \in \Sigma$, a (set) map 
$$
\alpha_v : \Xset(K_v)/\Aut(\bmu_p) = \Fset(K_v) \too \H(q_v),
$$
\item
for every $v \in \cP_2$, a bijection 
$$
\alpha_v : \Xset_\ram(K_v)/\Aut(\bmu_p) = \Fset_\ram(K_v) \too \H_\ram(q_v).
$$
\end{enumerate}
\end{defn}

\begin{rem}
Note that if $v \in \cP_2$ then 
$
|\Fset_\ram(K_v)| = p = |\H_\ram(q_v)|,
$
the first equality by local class field theory 
(since by definition $v\nmid p$ and $\bmu_p \subset K_v^\times$) 
and the second by Lemma \ref{countlem}(iii).

On the other hand,
if $v \in \cP_1$ then $\H_\ram(q_v)$ has exactly one element by Lemma \ref{countlem}(iii),  
and if $v \in \cP_0$ then either $H^1(K_v,T) = 0$ by Lemma \ref{countlem}(i), or 
$\bmu_p \not\subset K_v$ so $\Xset_\ram(K_v)$ is empty.  Thus if $v \in \cP_0 \cup \cP_1$ then 
there is a unique map $\Xset_\ram(K_v) \to \H_\ram(q_v)$.  That is why these maps do not need 
to be specified as part of the twisting data.
\end{rem}

If $\chi \in \Xset(K)$ and $v$ is a place of $K$, we let $\chi_v \in \Xset(K_v)$ 
denote the restriction of $\chi$ to $G_{K_v}$.

\begin{defn}
\label{ddef}
Let
$$
\D := \{\text{squarefree products of primes $\l \in \cP_1 \cup \cP_2$}\},
$$
and if $\d\in\D$ let $\d_1$ (resp., $\d_2$) be the product of all primes 
dividing $\d$ that lie in $\cP_1$ (resp.,  $\cP_2$), so $\d = \d_1\d_2$.
For every $\d \in \D$, define the {\em width} of $\d$ by 
$$
w(\d) := \sum_{\l\mid\d} w(\l) = |\{\l : \l\mid\d_1\}| + 2\cdot|\{\l : \l\mid\d_2\}|.
$$
Let $\Sigma(\d) := \Sigma \cup \{\l : \l \mid \d\} \subset \Sigma \cup \cP_1 \cup \cP_2$ and
\begin{multline*}
\Xset(\d) := \{\chi \in \Xset(K) : 
 \text{$\chi$ is ramified at all $\l$ dividing $\d$,} \\
    \text{and unramified outside of $\Sigma(\d) \cup \cP_0$}\}.
\end{multline*}  
Define a finite set
$$
\Gamma_\d := \prod_{v \in \Sigma}\Xset(K_v) \;\times 
   \prod_{\l \mid \d_2}\Xset_\ram(K_\l),
$$
and let $\eta_\d : \Xset(\d) \to \Gamma_\d$ and $\eta : \Xset(K) \to \Gamma_1$
denote the natural maps
$$
\eta_\d(\chi) := (\ldots,\chi_v,\ldots)_{v \in \Sigma(\d_2)}, \quad \eta(\chi) := (\ldots,\chi_v,\ldots)_{v \in \Sigma}.
$$
\end{defn}
  
Note that $\Xset(K_v)$ is a group, and $\Xset_\ram(K_v)$ is not a group but 
it is closed under multiplication by unramified characters.
Since $\Xset(\d)$ is the fiber over $\d$ of the map $\Xset(K) \to \D$ that sends $\chi$ to 
the part of its conductor supported on $\cP_1 \cup \cP_2$, we have 
 $\Xset(K) = \coprod_{\d\in\D} \Xset(\d)$.

\begin{defn}
\label{sstwist}
Given $T$, $\bq$, and twisting data as in Definition \ref{twistdata}, we define 
a Selmer structure $\cS(\gamma)$ for every $\d\in\D$ and 
$\gamma = (\gamma_v)_v \in\Gamma_\d$ as follows.
\begin{itemize}
\item
Let $\Sigma_{\cS(\gamma)} := \Sigma(\d)$.
\item
If $v \in \Sigma$ then let $H^1_{\cS(\gamma)}(K_v,T) := \alpha_v(\gamma_v)$.
\item
If $v \mid \d_1$, let $H^1_{\cS(\gamma)}(K_v,T)$ be the unique element of $\H_\ram(q_v)$.
\item
If $v \mid \d_2$, let $H^1_{\cS(\gamma)}(K_v,T) := \alpha_v(\gamma_v) \in \H_\ram(q_v)$.
\end{itemize}
If $\gamma\in\Gamma_\d$ we will also write $\Sel(T,\gamma) := H^1_{\cS(\gamma)}(K,T)$, and
if $\chi \in \Xset(\d)$ then we define
$$
\Sel(T,\chi) := \Sel(T,\eta_\d(\chi)).
$$
\end{defn}

\begin{rem}
\label{3.7}
It is clear from the definition that $\Sel(T,\chi)$ depends only on the 
extension of $K$ cut out by $\chi$, i.e., $\Sel(T,\chi) = \Sel(T,\chi^i)$ 
for all $i \in \Fp^\times$.  However, when we later count the twists $\Sel(T,\chi)$ 
with certain properties, it will be convenient to deal with $\Xset(K)$ 
rather than $\Fset(K)$ because $\Xset(K)$ is a group.  In any case 
the natural map $\Xset(K) \onto \Xset(K)/\Aut(\bmu_p) = \Fset(K)$ 
is $(p-1)$-to-one except for the single fiber consisting of the trivial character, 
so it is simple to go from counting results for $\Xset(K)$ to results for $\Fset(K)$. 
In particular, when $p = 2$ the natural map $\Xset(K) \to \Fset(K)$ 
is a bijection.
\end{rem}

\begin{rem}[Remarks about twisting data]
Our definition of twisting data is designed to ensure that for $v \notin \Sigma$, 
all subspaces $V \in \H_\ram(q_v)$ occur with equal frequency as we run 
over characters $\chi \in \Xset(K)$ that are ramified at $v$.  That fact is all we require 
to prove our results in sections \ref{pd2} and \ref{pdp}
about the rank statistics of $\Sel(T,\chi)$.  
In particular, the conclusions of Theorems \ref{sllem,p=2} and \ref{sllem} 
below do not depend on the choice of twisting data for $(T,\bq)$.

We will see in \S\ref{eex} that when $E$ is an elliptic curve over $K$, $p$ is 
a rational prime, and $T = E[p]$, then there is natural global \metabolic 
structure on $E[p]$ and natural twisting data such that for every $\chi \in \Xset(K)$, 
$\Sel(E[p],\chi)$ is a classical Selmer group of a twist of $E$ 
(an abelian variety twist, when $p>2$).  
An analogous statement should hold for more general (self-dual) motives and their 
Bloch-Kato $p$-Selmer groups, so our results below should also apply to  
Bloch-Kato Selmer groups in families of twists.

\end{rem}

\begin{defn}
\label{comp}
If $v \in \Sigma$ and $\psi, \psi' \in \Xset(K_v)$, define
$$
\com{\psi}{\psi'} := \dim_{\Fp}\alpha_v(\psi)/(\alpha_v(\psi)\cap\alpha_v(\psi')).
$$
\end{defn}

\begin{thm}
\label{kramer}
Suppose $\d \in \D$, $\gamma \in \Gamma_1$, and $\gamma' \in \Gamma_\d$.  Then
$$
\dim_{\F_p}\Sel(T,\gamma) - \dim_{\F_p}\Sel(T,\gamma') \equiv 
   \sum_{v \in \Sigma}\com{\gamma_v}{\gamma'_v} + w(\d) \pmod{2}.
$$
\end{thm}

\begin{proof}
We will deduce this from Theorem \ref{compsel}.  
Suppose $\l\mid\d$.  Then by definition $H^1_{\cS(\gamma)}(K_\l,T) = \Hu(K_\l,T)$ 
since $\l\notin\Sigma_{\cS(\gamma)}$, 
and $H^1_{\cS(\gamma')}(K_\l,T) \in \H_\ram(q_\l)$, so 
$$
H^1_{\cS(\gamma)}(K_\l,T) \cap H^1_{\cS(\gamma')}(K_\l,T) = 0.
$$
Also by definition we have $H^1_{\cS(\gamma)}(K_v,T) = \alpha_v(\gamma_v)$ if $v \in \Sigma$, 
and similarly for $\gamma'$.
Now applying Theorem \ref{compsel} with $\cS = \cS(\gamma)$ and $\cS' = \cS(\gamma')$ shows that
\begin{align*}
\dim_{\F_p}\Sel(T,&\gamma) - \dim_{\F_p}\Sel(T,\gamma') \\
  &\equiv \sum_{v \in \Sigma} \dim_{\Fp}\alpha_v(\gamma_v)/(\alpha_v(\gamma_v)\cap\alpha_v(\gamma_v')) + 
     \sum_{\l\mid\d}\dim_{\Fp}(\Hu(K_\l,T)) \\
  &= \sum_{v \in \Sigma}\com{\gamma_v}{\gamma'_v} + \sum_{\l\mid\d} w(\l) \pmod{2},
\end{align*}
using that if $\l\in\cP_1 \cup \cP_2$ then $\dim_{\Fp}\Hu(K_\l,T) = w(\l)$ 
by Lemma \ref{countlem}(ii).  This proves the theorem.
\end{proof}

\begin{cor}
\label{kramercor}
Suppose $\d\in\D$ and $\chi\in\Xset(\d)$.  Then 
$$
\dim_{\Fp}\Sel(T,\chi) \equiv \dim_{\Fp}\Sel(T,\eta(\chi)) 
   + w(\d) \pmod{2}.
$$
\end{cor}

\begin{proof}
Let $\gamma = \eta(\chi)$ and $\gamma' = \eta_\d(\chi)$.
If $v \in \Sigma$ then $\eta(\chi)_v = \chi_v = \eta_\d(\chi)_v$, so $\com{\gamma_v}{\gamma'_v} = 0$.
Now the corollary follows from Theorem \ref{kramer}.
\end{proof}

\section{Example: twists of elliptic curves}
\label{eex}

Fix for this section an elliptic curve $E$ defined over $K$, a prime $p$, 
and let $T := E[p]$.  We will show that this $T$ 
comes equipped with the extra structure that we require, and that with an appropriate 
choice of twisting data, the Selmer groups $\Sel(E[p],\chi)$ are classical $p$-Selmer groups 
of twists of $E$.

The module $T = E[p]$ satisfies the hypotheses of \S\ref{notation}, with the pairing 
\eqref{weilpair} given by the Weil pairing. Let $\Sigma$ be a finite set of places 
of $K$ containing all archimedean places, all places above $p$, and all primes where 
$E$ has bad reduction.
Let $\O$ denote the ring of integers of the cyclotomic field of $p$-th roots of unity, 
and $\p$ the (unique) prime of $\O$ above $p$.

If $p > 2$, there is a unique global \metabolic structure $\bq_E = (q_{E,v})$ on $E[p]$ (Lemma \ref{2.4}).  
For general $p$, there is a canonical 
global \metabolic structure $\bq_E$ on $E[p]$ constructed from the Heisenberg group, 
see \cite[\S4]{poonenrains}.  We recall this construction below when $p=2$, in the proof of 
Lemma \ref{sameT}(ii).

We next define twisting data for $(E[p],\Sigma,\bq_E)$ in the sense of Definition \ref{twistdata}.

\begin{defn}
\label{etwist}
Suppose $\chi\in\Xset(K)$ (or $\chi \in \Xset(K_v)$) is nontrivial.  
If $p=2$ we let $E^\chi$ denote the quadratic twist of $E$ by $\chi$.
For general $p$, let $F$ denote the cyclic
extension of $K$ (resp., $K_v$) of degree $p$ corresponding to $\chi$, 
and let $E^\chi$ denote the abelian variety 
denoted $E_F$ in \cite[Definition 5.1]{mrs}.  

Concretely, if $\chi \in \Xset(K)$ and $\chi\ne \one_K$ then $E^\chi$ is an abelian variety of 
dimension $p-1$ over $K$, defined to be the kernel of the canonical map
$$
\Res^F_K(E) \too E
$$
where $\Res^F_K(E)$ denotes the Weil restriction of scalars of $E$ from $F$ to $K$.
The character $\chi$ induces an inclusion $\O \subset \End_K(E^\chi)$ (see \cite[Theorem 5.5(iv)]{mrs}).  

For $\chi \in \Xset(K)$, let $\bq_{E^\chi} = (q_{E^\chi,v})$ be the unique global \metabolic structure 
on $E^\chi[\p]$ if $p>2$, and if $p=2$ 
we let $\bq_{E^\chi}$ be the canonical global \metabolic structure on the elliptic curve $E^\chi$.

\end{defn}

If $p=2$, then the two definitions above of $E^\chi$ agree, with $\O = \Z$, and $\p = 2$.

\begin{lem}
\label{sameT}
\begin{enumerate}
\item
There is a canonical $G_K$-isomorphism $E^\chi[\p] \cong E[p]$.
\item
The isomorphism of (i) identifies $q_{E^\chi,v}$ with $q_{E,v}$ for every $v$ 
and every $\chi\in\Xset(K_v)$.
\end{enumerate}
\end{lem}

\begin{proof}
Assertion (i) follows directly from the definition of quadratic twist when $p = 2$ 
(or see the proof of (ii) below).  For general $p$, 
see \cite[Theorem 2.2(iii)]{mrs} or \cite[Proposition 4.1]{MRvisibility}.

The isomorphism of (i) identifies the Weil pairings on $E^\chi[\p]$ and $E[p]$, 
and hence it identifies the local Tate pairings on $H^1(K_v,E^\chi[\p])$ and $H^1(K_v,E[p])$ for every $v$.
Thus when $p > 2$, assertion (ii) follows from the uniqueness of the 
Tate quadratic form on $T$ (Lemma \ref{2.4}).
When $p=2$, we use an explicit construction of the quadratic form $q_{E,v}$.  
Let $\bar{K}_v(E)$ denote the function field of $E$ over $\bar{K}_v$.  Following 
\cite[Proposition 1.32]{cfoss} we define the Heisenberg (or theta) group 
$$
\Theta_E := \{(f,P) \in \bar{K}_v(E) \times E[2] : \text{the divisor of $f$ is $2[P]-2[O]$}\}
$$
with group law
$$
(f,P)\cdot(g,Q) := (\tau_Q^*(f)g,P+Q)
$$
where $\tau_Q$ is translation by $Q$ on $E$.  
The projection $\Theta_E \to E[2]$ induces an exact sequence
\begin{equation}
\label{tses}
1 \too \bar{K}_v^\times \too \Theta_E \too E[2] \too 0.
\end{equation}
We view $\Theta_E$ as an extension of $E[2]$ by $\G_m$, functorial in the sense that 
if $E'$ is another elliptic curve over $K_v$ and $\lambda : E \to E'$ is an isomorphism 
over $\bar{K}_v$, then $\lambda$ induces an isomorphism $\lambda^* : \Theta_{E'} \to \Theta_E$ 
over $\bar{K}_v$ (that commutes with \eqref{tses} in the obvious sense).  It is easy to see that the map
$$
\Isom(E,E') \too \Isom(\Theta_{E'},\Theta_E)
$$
defined by  $\lambda \mapsto \lambda^*$ is a $G_{K_v}$-equivariant homomorphism.

With this notation, $q_{E,v} : H^1(K_v,E[2]) \to H^2(K_v,\bar{K}_v^\times) \subset \Q/\Z$ is the connecting 
map of the long exact sequence of (nonabelian) Galois cohomology attached to \eqref{tses}. 

Fix an isomorphism $\lambda : E \to E^\chi$ defined over the quadratic field cut out by 
$\chi$.  For every  $\sigma \in  G_{K_v}$ we have $\lambda^\sigma = \lambda \circ [\chi(\sigma)]$, where 
$[\chi(\sigma)] : E \to E$ is multiplication by $\chi(\sigma) = \pm1$.
Thus the isomorphism $\lambda^* : \Theta_{E^\chi} \to \Theta_E$ induced by $\lambda$ satisfies 
$(\lambda^*)^\sigma = (\lambda^\sigma)^* = [\pm1]^*\circ\lambda^*$.

Clearly $[-1]$ acts trivially on $E[2]$.  
Suppose $f \in\bar{K}_v(E)$ has divisor $2[P]-2[O]$ with $P \in E[2] - O$.  
If we fix a Weierstrass model of $E$ with coordinate functions $X, Y$, then $f$ is a 
constant multiple of $X - X(P)$, so $f \circ [-1] = f$.  Hence $[-1]^*$ is the identity 
on $\Theta_E$, so in fact $(\lambda^*)^\sigma = \lambda^*$ for every $\sigma \in  G_{K_v}$.
Hence $\Theta_{E^\chi}$ and $\Theta_E$ are isomorphic over $K_v$ as extensions of $E[2]$,  
so by the definition above we have $q_{E^\chi,v} = q_{E,v}$.
\end{proof}

\begin{defn}
\label{etwda}
Let $\pi$ denote any generator of the ideal $\p$ of $\O$.
If $v$ is a place of $K$ and $\chi \in \Xset(K_v)$, define $\alpha_v(\chi)$ to be the image of the 
composition of the Kummer ``division by $\pi$'' map with the isomorphism of Lemma \ref{sameT}(i)
$$
\alpha_v(\chi) := \image\biggl(E^\chi(K_v)/\p E^\chi(K_v) \hookto H^1(K_v,E^\chi[\p]) 
   \isom H^1(K_v,E[p])\biggr).
$$
\end{defn}

Note that $\alpha_v(\chi)$ is independent of the choice of generator $\pi$.

\begin{lem}
\label{lagrangian}
For every place $v$ and $\chi\in\Xset(K_v)$, we have $\alpha_v(\chi) \in \H(q_{E,v})$.
\end{lem}

\begin{proof}
If $p=2$, then \cite[Proposition 4.10]{poonenrains} shows for every $v$ that 
the image of $E^\chi(K_v)/2E^\chi(K_v)$ in $H^1(K_v,E^\chi[2])$ is a Lagrangian subspace for $q_{v,E^\chi}$.  
If $p>2$, then \cite[Proposition A.7]{MRvisibility} 
(together with Lemma \ref{2.4}) shows that $\alpha_v(\chi)$ is a Lagrangian subspace 
for the (unique) Tate quadratic form on $H^1(K_v,E^\chi[\p])$, and hence for $q_{v,E^\chi}$.  
Now the lemma follows from Lemma \ref{sameT}(ii).
\end{proof}

As in Definition \ref{comp}, let 
$\com{\chi}{\chi'} := \dim_{\Fp}(\alpha_v(\chi)/(\alpha_v(\chi)\cap\alpha_v(\chi')))$.

\begin{lem}
\label{prekramer}
Suppose $v$ is a place of $K$, and $\chi \in \Xset(K_v)$.  Let $F/K_v$ be the 
cyclic extension cut out by $\chi$.  Then 
$$
\com{\one_v}{\chi} = \dim_{\Fp}E(K_v)/N_{F/K_v}E(F).
$$
\end{lem}

\begin{proof}
When $p=2$, this is due to Kramer \cite[Proposition 7]{kramer}.  
For general $p$ this is \cite[Corollary 5.3]{MRvisibility}.  
(That result is stated only for $p>2$, but the 
proof for $p=2$ is the same.)  
\end{proof}

\begin{lem}
\label{ttd}
Suppose $p=2$, $v \in \Sigma$, and $\psi \in \Xset(K_v)$.  
Let 
$$
\alpha_v^\psi : \Xset(K_v) \too \H(q_{E^\psi,v}) = \H(q_{E,v})
$$
and 
$\compsi{\one_v}{\chi} := \dim_{\F_2}(\alpha_v^\psi(\one_v)/(\alpha_v^\psi(\one_v)\cap\alpha_v^\psi(\chi)))$
be as defined in Definitions \ref{etwda} and \ref{comp}, respectively, for $E^\psi$ instead of $E$.
Then
$$
\compsi{\one_v}{\chi} = \com{\psi}{\chi\psi} \equiv \com{\one_v}{\psi} + \com{\one_v}{\chi\psi} \pmod{2}.
$$
\end{lem}

\begin{proof}
It follows directly from Definition \ref{etwda} that 
$\alpha_v^\psi(\chi) = \alpha_v(\chi\psi)$ for every $\chi\in\Xset(K_v)$.
This proves the equality, and the congruence follows from Corollary \ref{ell3} 
applied to the Lagrangian subspaces
$\alpha_v(\one_v)$, $\alpha_v(\psi)$, and $\alpha_v(\psi\chi)$.
\end{proof}

\begin{lem}
\label{4.6}
Suppose $p > 2$, $v\in\cP_2$, and $\chi\in\Xset(K_v)$ is nontrivial.  If $F$ is the 
cyclic extension of $K_v$ corresponding to $\chi$, then
$$
\alpha_v(\chi) = \Hom(\Gal(F/K_v),E[p]) \subset \Hom(G_{K_v},E[p]) = H^1(K_v,E[p]).
$$
\end{lem}

\begin{proof}
Let $G := \Gal(F/K_v)$.  Fix a generator $\sigma$ of $G$, and let $\pi = 1 - \chi(\sigma) \in \O$, 
so $\pi\O = \p$.
Let $\I := (\sigma-1)\Z[G]$ be the augmentation ideal of $\Z[G]$.
By \cite[Theorem 2.2(ii)]{mrs}, we have an isomorphism of $\O[G_{K_v}]$-modules 
$$
E^\chi[p] = \I \otimes_\Z E[p] 
$$
where $\gamma\in G_{K_v}$ acts by $\gamma^{-1} \otimes \gamma$ on $\I \otimes E[p]$, 
and $\pi$ acts as multiplication by $(1-\sigma)\otimes 1$ on $\I \otimes E[p]$.
Since $v \in \cP_2$, $G_{K_v}$ acts trivially on $E[p]$, and $G_F$ acts trivially on 
both $\I$ and $E[p]$.  Hence 
\begin{multline}
\label{stable1}
E^\chi(K_v)[p] = E^\chi[p]^{G_{K_v}} = (\I \otimes_\Z E[p])^{G_{K_v}} \\
   = (\I \otimes_\Z E[p])^{\sigma=1} = (\I \otimes_\Z E[p])^{\pi=0} =  E^\chi[\p], 
\end{multline}
\begin{equation}
\label{stable2}
E^\chi(F)[p] = E^\chi[p]^{G_F} = (\I \otimes_\Z E[p])^{G_{F}} = \I \otimes_\Z E[p] = E^\chi[p].
\end{equation}
Since $p>2$, we have $\p^2 \mid p$, so it follows from \eqref{stable1} that 
$E^\chi(K_v)[p^\infty] = E^\chi[\p]$.  Therefore since $v \nmid p\infty$, we have
$E^\chi(K_v) \cong E^\chi[\p] \times B$ with a profinite abelian group $B$ such that $p B = B$, 
so we deduce from \eqref{stable2} that $E^\chi(K_v) \subset \pi E^\chi(F)$.  
Identifying $H^1(K_v,E[p])$ with $\Hom(G_{K_v},E[p])$, it follows from Definition \ref{etwda} that 
if $c \in \alpha_v(\chi) \subset \Hom(G_{K_v},E[p])$, then $c(G_F) = 0$.  
Thus $\alpha_v(\chi) \subset \Hom(\Gal(F/K_v),E[p])$.  
By Lemma \ref{countlem}(ii) we have 
$\dim_{\Fp}\alpha_v(\chi) = 2 = \dim_{\Fp}\Hom(\Gal(F/K_v),E[p])$, which completes the proof.
\end{proof}

\begin{prop}
The maps $\alpha_v$ of Definition \ref{etwda}, for $v \in \Sigma$ and $v \in \cP_2$, 
give twisting data as in Definition \ref{twistdata}.
\end{prop}

\begin{proof}
It follows from the definition that $\alpha_v(\chi)$ depends only on the extension of 
$K_v$ cut out by $\chi$.

By Lemma \ref{lagrangian}, $\alpha_v(\chi) \in \H(q_{E,v})$ for every 
$v$ and every $\chi\in\Xset(K_v)$.
Thus $\alpha_v$ satisfies Definition \ref{twistdata}(i) for $v \in \Sigma$.

Now suppose that $v \in \cP_2$.  
If $\chi \in \Xset_\ram(K_v)$, then $\alpha_v(\chi)\cap\Hu(K_v,T) = 0$  
by \cite[Lemma 2.11]{MRstablerank}, so  $\alpha_v(\chi) \in \H_\ram(q_{E,v})$. 
To complete the proof of the proposition we need only show that the map 
$\alpha_v : \Xset_\ram(K_v)/\Aut(\bmu_p) \to \H_\ram(q_{E,v})$ is a bijection.  
Since 
$$
|\Xset_\ram(K_v)/\Aut(\bmu_p)| = p = |\H_\ram(q_{E,v})|
$$ 
by local class field theory and Lemma \ref{countlem}(iii), 
we only need to show the injectivity of $\alpha_v$, and when $p > 2$ 
this follows from Lemma \ref{4.6}.

Suppose $p = 2$.  Let $\psi,\chi \in \Xset(K_v)$ be a ramified and nontrivial unramified character, 
respectively.  Then $\Xset_\ram(K_v) = \{\psi,\chi\psi\}$, so we need only show that 
$\alpha_v(\psi) \ne \alpha_v(\chi\psi)$.

Let $F$ be the unramified quadratic extension of $K_v$.  Since $v \in \cP_2$, 
we have that $v \nmid 2$, $E$ has good reduction at $v$, and $G_{K_v}$ acts trivially on $E[2]$.  
Since $\psi$ is ramified over $F$, $E^\psi$ has additive reduction over $F$ above $v$.  
Tate's algorithm \cite{tatealg} shows that 
$E^\psi(F)[2^\infty] = E^\psi(F)[2] = E^\psi(K_v)[2] = E[2]$, and so 
$E^\psi(K_v)$ and $E^\psi(F)$ are each isomorphic to the product of the Klein $4$-group $E[2]$ 
with profinite abelian groups of odd order.
Hence $\N_{F/K_v}E^\psi(F) = 2E^\psi(K_v)$, 
so by Lemmas \ref{prekramer} and \ref{ttd}
$$
2 = \compsi{\one_v}{\chi} = \com{\psi}{\chi\psi} 
   = \dim_{\F_2}(\alpha_v(\psi)/(\alpha_v(\psi) \cap \alpha_v(\chi\psi)))
$$
and in particular $\alpha_v(\psi) \ne \alpha_v(\chi\psi)$.
\end{proof}

\begin{prop}
\label{4.9}
With the twisting data of Definition \ref{etwda}, and any generator $\pi$ of $\p$, 
for $\chi \in \Xset(K)$ we have 
that $\Sel(E[p],\chi) \cong \Sel_\pi(E^\chi/K)$, the usual $\pi$-Selmer group of $E^\chi/K$.
In particular when $p=2$, $\Sel(E[2],\chi) = \Sel_2(E^\chi/K)$ is 
the classical $2$-Selmer group of $E^\chi/K$.
\end{prop}

\begin{proof}
Let $\d\in\D$ be such that $\chi\in\Xset(\d)$.
By definition
$$
\Sel_\pi(E^\chi/K) \cong \{c \in H^1(K,E[p]) : \text{$c_v \in \alpha_v(\chi)$ for every $v$}\}
$$
with $\alpha_v(\chi)$ as in Definition \ref{etwda}.  Thus we need to show that 
$H^1_{\cS(\gamma)}(K_v,T) = \alpha_v(\chi)$ for every $v$, where 
$\gamma := \eta_\d(\chi) \in \Gamma_\d$.  

If $v \in \Sigma$, or if $v\mid\d$ and $v\in\cP_2$, then this is the definition of 
$H^1_{\cS(\gamma)}(K_v,T)$.  
If $v\in\cP_0$ and $\chi$ is ramified at $v$, then $H^1(K_v,T) = 0$ by Lemma \ref{countlem}(i), 
and if $v \notin \Sigma$ and $\chi$ is unramified at $v$, then 
$\alpha_v(\chi) = \Hu(K_v,T)$ by \cite[Lemma 4.1]{cassels}, so in those cases we also have 
$H^1_{\cS(\gamma)}(K_v,T) = \alpha_v(\chi)$.

It remains only to check those $v$ such that $v\mid\d$ and $v\in\cP_1$.   
In that case $\alpha_v(\chi)\cap\Hu(K_v,T) = 0$ by 
\cite[Lemma 2.11]{MRstablerank}, so $\alpha_v(\chi) \subset \H_\ram(q_{E,v})$.  
But in this case $|\H_\ram(q_{E,v})| = 1$ by Lemma \ref{countlem}(iii), and 
$H^1_{\cS(\gamma)}(K_v,T)$ is the unique element of 
$\H_\ram(q_{E,v})$ by definition, so $H^1_{\cS(\gamma)}(K_v,T) = \alpha_v(\chi)$ in this case 
also.  This completes the proof.
\end{proof}

\begin{rem}
Definition \ref{etwist} of $E^\chi$ shows that $E^\chi$ depends only on the field cut out by $\chi$, 
not on the choice of character $\chi$ itself.  As mentioned in Remark \ref{3.7}, it is easier to count 
characters of order $p$ than cyclic extensions of degree $p$, because the set of characters is a group.

If $F/K$ is the cyclic extension cut out by $\chi \ne \one_K$, then the short exact sequence 
$0 \to E^\chi \to \Res^F_K(E) \to E \to 0$ of Definition \ref{etwist} gives an identity of Mordell-Weil 
ranks 
$$
\rk(E(F)) = \rk(\Res^F_K(E)(K)) = \rk(E(K)) + \rk(E^\chi(K)).
$$
In particular if $\Sel(E[p],\chi) = 0$, then by Proposition \ref{4.9} we have 
$\rk(E(F)) = \rk(E(K))$.    
\end{rem}

\section{Local and global characters}
\label{lgc}
For the rest of this paper we
fix $T$ and $\Sigma$ as in \S\ref{notation}, a global \metabolic structure $\bq$ on $T$ 
as in Definition \ref{ahsdef}, 
and twisting data as in Definition \ref{twistdata}.  
Recall that $K(T)$ is the field of definition of the elements of $T$, i.e., 
the fixed field in $\Kb$ of $\ker (G_K \to \Aut(T))$.

For the rest of this paper we assume also that
\begin{equation}
\label{h2a} 
\Pic(\O_{K,\Sigma}) = 0,
\end{equation}
and
\begin{equation}
\label{h2b}
\O_{K,\Sigma}^\times/(\O_{K,\Sigma}^\times)^p \too \prod_{v \in \Sigma} K_v^\times/(K_v^\times)^p \quad
   \text{is injective},
\end{equation}
where $\O_{K,\Sigma}$ is the ring of $\Sigma$-integers of $K$, i.e., the 
elements that are integral at all $\l \notin \Sigma$.

\begin{lem}
Conditions \eqref{h2a} and \eqref{h2b} can always be satisfied by enlarging $\Sigma$ if necessary.
\end{lem}

\begin{proof}
First, enlarge $\Sigma$ if necessary by adding primes $\l$ whose classes generate 
the ideal class group of $\O_K$.  After this we will have \eqref{h2a}.  
Further increases will preserve this condition.  

Let $f_\Sigma$ denote the natural map 
$\O_{K,\Sigma}^\times/(\O_{K,\Sigma}^\times)^p \to \prod_{v \in \Sigma} K_v^\times/(K_v^\times)^p$, 
and suppose $u \in \ker(f_\Sigma)$ is nontrivial.
The kernel of $K^\times/(K^\times)^p \to K(\bmu_p)^\times/(K(\bmu_p)^\times)^p$ 
is $H^1(K(\bmu_p)/K,\bmu_p) = 0$, so $u \notin (K(\bmu_p)^\times)^p$ 
and $[K(\bmu_p,u^{1/p}):K(\bmu_p)] = p$.  
Let $\l$ be a prime of $K$ whose Frobenius automorphism in $\Gal(K(\bmu_p,u^{1/p})/K)$ 
has order $p$.  Then $u \notin (K_\l^\times)^p$, so $f_{\Sigma\cup\{\l\}}(u) \ne 1$.

Let $\Sigma' := \Sigma \cup \{\l\}$.  Since $\Pic(\O_\Sigma) = 0$, there is a 
$\lambda \in \O_{K,\Sigma'}$ such that $\ord_\l(\lambda) = 1$.  Thus 
$\O_{K,\Sigma'}^\times = \O_{K,\Sigma}^\times \times \ld\lambda\rd$, 
where $\ld\lambda\rd$ is the infinite cyclic group generated by $\lambda$.  
The map $\ld\lambda\rd/\ld\lambda^p\rd \to K_\l^\times/(K_\l^\times)^p$ is injective, 
so $\ker(f_{\Sigma'}) \subsetneq \ker(f_\Sigma)$
and the inclusion is strict because $\ker(f_\Sigma)$ contains $u$ and $\ker(f_{\Sigma'})$ does not.
Replacing $\Sigma$ by $\Sigma'$, we can continue in this way until $\ker(f_\Sigma) = 1$, 
i.e., until \eqref{h2b} holds.
\end{proof} 

\begin{lem}
\label{4.3}
Define the subgroup $\A \subset K^\times/(K^\times)^p$ by 
$$
\A := \ker(K^\times/(K^\times)^p \to K(T)^\times/(K(T)^\times)^p).
$$
Then there is a canonical isomorphism 
$$
\A \isom \Hom(\Gal(K(T)/K(\bmu_p)),\bmu_p)^{\Gal(K(T)/K)},
$$
and $\A$ is cyclic, generated by an element $\Delta \in \O_{K,\Sigma}^\times$.
\end{lem}

\begin{proof}
The inflation-restriction sequence of Galois cohomology, together with 
the fact that $H^1(F,\bmu_p) = F^\times/(F^\times)^p$ for every field $F$ of characteristic 
different from $p$, shows that 
$$
\A = \ker(H^1(K,\bmu_p) \to H^1(K(T),\bmu_p)) = H^1(K(T)/K,\bmu_p).
$$
Since $[K(\bmu_p):K]$ is prime to $p$, the Hochschild-Serre spectral sequence gives an isomorphism
\begin{multline*}
\A = H^1(K(T)/K,\bmu_p) \isom H^1(K(T)/K(\bmu_p),\bmu_p)^{\Gal(K(T)/K)} \\
    = \Hom(\Gal(K(T)/K(\bmu_p)),\bmu_p)^{\Gal(K(T)/K)}.
\end{multline*}
Since $\Gal(K(T)/K(\bmu_p))$ is isomorphic to a subgroup of 
$\SL_2(\Fp)$, we see that 
$\Hom(\Gal(K(T)/K(\bmu_p)),\bmu_p)$ has order $1$ or $p$.  Thus $\A$ is cyclic.

Let $I_{K,\Sigma}$ (resp., $I_{K(T),\Sigma}$) denote the group of fractional ideals 
of $K$ (resp., $K(T)$) prime to $\Sigma$.  We have a commutative diagram 
$$
\xymatrix@C=15pt{
1 \ar[r] & \O_{K,\Sigma}^\times/(\O_{K,\Sigma}^\times)^p \ar[r] & K^\times/(K^\times)^p \ar[r] \ar[d]
   & I_{K,\Sigma}/I_{K,\Sigma}^p \ar[r]\ar[d] & 1 \\
&& K(T)^\times/(K(T)^\times)^p \ar[r] & I_{K(T),\Sigma}/I_{K(T),\Sigma}^p
}
$$
in which the top row is exact by \eqref{h2a}.  Since $K(T)/K$ is unramified outside of $\Sigma$, 
the right-hand vertical map is injective, so $\A$ (the kernel of the left-hand vertical map) 
is contained in the image of of $\O_{K,\Sigma}^\times$.  This completes the proof of the lemma.
\end{proof}

\begin{lem}
Suppose $p \le 3$, $E$ is an elliptic curve over $K$, and $T = E[p]$.  Then we can take the element 
$\Delta$ of Lemma \ref{4.3} to be the discriminant $\Delta_E$ of (any model of) $E$.
\end{lem}

\begin{proof}
Since $\Sigma$ contains all primes where $E$ has bad reduction, 
we have $\Delta_E \in \O_{K,\Sigma}^\times \cdot(K^\times)^{12}$.
By Lemma \ref{4.3}, $\A = \{1\}$ if $p \nmid [K(E[p]):K]$.

Suppose first that $p = 2$.  Then $\Delta_E \in (K(E[2])^\times)^2$ and 
$[K(E[2]):K(\sqrt{\Delta_E})]$ divides $3$.
Thus $\Delta_E \in (K^\times)^2$ if and only if $2 \nmid [K(E[2]):K]$, so $\Delta_E$ generates $\A$.

Similarly, when $p = 3$, computing the discriminant of the universal elliptic curve with full level $3$ 
structure (see for example \cite[\S1.1]{modpreps}) shows that $\Delta_E \in (K(E[3])^\times)^3$ 
and that $[K(E[3]):K(\bmu_3,\sqrt[3]{\Delta_E})]$ divides $8$, so again $\Delta_E$ generates $\A$.
\end{proof}

Fix once and for all a $\Delta \in \O_{K,\Sigma}^\times$ as in Lemma \ref{4.3}.  
Recall (Definition \ref{ddef}) that $\Gamma_1 := \prod_{v \in \Sigma} \Xset(K_v)$, 
and more generally $\Gamma_\d := \prod_{v \in \Sigma}\Xset(K_v) \;\times 
   \prod_{\l \mid \d_2}\Xset_\ram(K_\l)$
for $\d\in\D$.
For each $v$, local class field theory identifies $\Xset(K_v)$ 
with $\Hom(K_v^\times,\bmu_p)$.

\begin{defn}
\label{gpmdef}
Define a ``sign'' homomorphism 
$
\sgn : \Gamma_1 \to \bmu_p
$
by 
$$
\sgn(\ldots,\gamma_v,\ldots) := \prod_{v \in \Sigma} \gamma_v(\Delta).
$$
Composing with the natural maps $\Gamma_\d \to \Gamma_1$ and $\Xset(K) \to \Gamma_1$, we will 
extend $\sgn$ to $\Gamma_\d$ for every $\d\in\D$, and to $\Xset(K)$.
\end{defn}

\begin{lem}
\label{5.6}
Suppose $\A$ is nontrivial, i.e., $\Delta \notin (K^\times)^p$.
\begin{enumerate}
\item
If $\l \in \cP_2$ and $\chi_\l \in \Xset(K_\l)$, then $\chi_\l(\Delta) = 1$.
\item
If $\l\in\cP_1$ and $\chi_\l\in\Xset(K_\l)$, then $\chi_\l(\Delta) = 1$ if and only if 
$\chi_\l$ is unramified. 
\item
If $p=2$, $\l\in\cP_0$, and $\chi_\l\in\Xset(K_\l)$, then $\chi_\l(\Delta) = 1$. 
\end{enumerate}
\end{lem}

\begin{proof}

By Lemma \ref{4.2}(i), if $\l\in\cP_2$ then $\Frob_\l$ fixes $K(T)$, so 
$\Frob_\l$ fixes $\Delta^{1/p}$, so $\Delta \in (K_\l^\times)^p$.  
This proves (i). 
Similarly, if $p=2$ and $\l\in\cP_0$ then Lemma \ref{4.2}(iii) shows that 
$\Frob_\l \in \Gal(K(T)/K)$ has order $3$.  
Thus again $\Frob_\l$ fixes $\sqrt{\Delta}$, so $\Delta \in (K_\l^\times)^2$.  
This proves (iii).  

If $\l\in\cP_1$ then Lemma \ref{4.2}(ii) shows that 
$\Frob_\l \in \Gal(K(T)/K)$ has order exactly $p$, and in particular $\Frob_\l \in \Gal(K(T)/K(\bmu_p))$.  
But since $\Delta \notin (K^\times)^p$, the degree $[K(T):K(\bmu_p,\Delta^{1/p})]$ is prime to $p$. 
Thus $\Frob_\l$ does not fix $\Delta^{1/p}$, so $\Delta \notin (K_\l^\times)^p$.  Therefore $\Delta$ 
generates $\O_\l^\times/(\O_\l^\times)^p \cong \Z/p\Z$, where $\O_\l$ is the ring of integers of $K_\l$.  
It follows that for $\chi_\l \in \Xset(K_\l)$, we have $\chi_\l(\Delta) = 1$ if and only if 
$\chi_\l(\O_\l^\times) = 1$.  This is (ii).
\end{proof}

\begin{lem}
\label{elem}
Suppose $G$ and $H$ are abelian groups, and $J \subset G \times H$ is a subgroup.  
Let $\pi_G$ and $\pi_H$ denote the projection maps from $G \times H$ to $G$ and $H$, respectively.
Let $J_0 := \ker(J \map{\pi_G} G/G^p)$.
\begin{enumerate}
\item
The image of the natural map $\Hom((G \times H)/J,\bmu_p) \to \Hom(H,\bmu_p)$ is 
$\Hom(H/\pi_H(J_0),\bmu_p)$.
\item

If $J/J^p \to G/G^p$ is injective, 
then $\Hom((G \times H)/J,\bmu_p) \to \Hom(H,\bmu_p)$ 
is surjective.
\end{enumerate}
\end{lem}

\begin{proof}
We have an exact sequence of $\Fp$-vector spaces 
$$
0 \too \pi_H(J_0) H^p/H^p \too H/H^p \too (G \times H)/J(G \times H)^p.
$$
Assertion (i) follows by applying $\Hom(\;\cdot\;,\bmu_p)$, and (ii) follows directly from (i).
\end{proof}

\begin{lem}
\label{precft}
\begin{enumerate}
\item
The natural map 
$\O_{K,\Sigma}^\times/(\O_{K,\Sigma}^\times)^p \to \prod_{\l\notin\Sigma}\O_\l^\times/(\O_\l^\times)^p$
is injective.
\item
The kernel of the natural map 
$\O_{K,\Sigma}^\times/(\O_{K,\Sigma}^\times)^p \to \prod_{\l\in\cP_2}\O_\l^\times/(\O_\l^\times)^p$
is $\A$.
\end{enumerate}
\end{lem}

\begin{proof}
Suppose $\alpha\in\O_{K,\Sigma}^\times$, $\alpha \notin (\O_{K,\Sigma}^\times)^p$.  
Then $\alpha \notin (K(\bmu_p)^\times)^p$.
If $\l\notin\Sigma$ is any prime whose Frobenius in $\Gal(K(\bmu_p,\alpha^{1/p})/K(\bmu_p))$ is nontrivial,  
then $\alpha \in \O_\l^\times$ but $\alpha \notin (\O_\l^\times)^p$.  Thus $\alpha$ is not in the 
kernel of the map of (i), and (i) follows.

Now suppose $\alpha \notin \A$.  Then $\alpha^{1/p} \notin K(T)$, 
so we can choose a nontrivial automorphism $\sigma \in \Gal(K(T,\alpha^{1/p})/K(T))$.  
Suppose $\l\notin\Sigma$ is a prime whose Frobenius in $\Gal(K(T,\alpha^{1/p})/K)$ is $\sigma$.  
Then $\l\in\cP_2$ by Lemma \ref{4.2}(i), and 
$\alpha \in \O_\l^\times$ but $\alpha \notin (\O_\l^\times)^p$.  This shows that 
the kernel of the map of (ii) is contained in $\A$, and $\A$ is contained in the kernel 
by Lemma \ref{5.6}(i).
\end{proof}

\begin{prop}
\label{cft}
\begin{enumerate}
\item
The natural homomorphism $\Xset(K) \to \Gamma_1$ is surjective.
\item
If $\l\notin\Sigma$ and $\bmu_p \subset K_\l^\times$, then 
there is a $\chi\in\Xset(K)$ ramified at $\l$ and unramified outside of $\Sigma$ and $\l$.
\item
There is a finite subgroup of $\Xset(K)$, containing only characters 
unramified outside of $\Sigma$ and $\cP_2$, whose image in $\Gamma_1$ is 
$\ker(\sgn)$.
\end{enumerate}
\end{prop}

\begin{proof}
Let $\iK$ denote the ideles of $K$.  Global class field theory and \eqref{h2a} show that 
\begin{equation}
\label{ck}
\Xset(K) = \Hom(\iK/K^\times,\bmu_p) 
   = \textstyle\Hom((\prod_{v \in \Sigma}K_v^\times \times 
      \prod_{\l\notin\Sigma}\O_\l^\times)/\O_{K,\Sigma}^\times,\bmu_p).
\end{equation}

For (i), we apply Lemma \ref{elem}(ii) with 
$$
G := \prod_{\l\notin\Sigma}\O_\l^\times, \quad
   H := \prod_{v \in \Sigma}K_v^\times, \quad J := \O_{K,\Sigma}^\times.
$$  
Then $J/J^p \to G/G^p$ is injective by Lemma \ref{precft}(i), 
so Lemma \ref{elem}(ii) and \eqref{ck} show that
$$
\textstyle
\Xset(K) \onto \Hom(\prod_{v \in \Sigma}K_v^\times,\bmu_p) = \Gamma_1
$$
is surjective.

For (ii), we apply Lemma \ref{elem}(ii) with 
$$
G := \prod_{v \in \Sigma}K_v^\times, \quad 
   H := \prod_{v\notin\Sigma}\O_v^\times, \quad J := \O_{K,\Sigma}^\times.
$$
Assumption \eqref{h2b} says that the map $J/J^p \to G/G^p$ is injective, 
so by Lemma \ref{elem}(ii) and \eqref{ck} we have that
\begin{equation}
\label{ck2}
\textstyle
\Xset(K) \onto \Hom(\prod_{v\notin\Sigma}\O_v^\times,\bmu_p)
\end{equation}
is surjective.  If $\bmu_p \subset K_\l^\times$ then we can fix an element 
$\psi \in \Hom(\prod_{v\notin\Sigma}\O_v^\times,\bmu_p)$ such that $\psi(\O_\l^\times) \ne 1$ 
but $\psi(\O_v^\times) = 1$ if $v \ne \l$, and let $\chi \in \Xset(K)$ be a character that maps to $\psi$ 
under \eqref{ck2}.  
Then $\chi$ satisfies (ii).

For (iii), we apply Lemma \ref{elem}(i) with 
$$
G := \prod_{\l\in\cP_2}\O_\l^\times, \quad 
   H := \prod_{v \in \Sigma}K_v^\times \times \prod_{\l\in\cP_0\cup\cP_1}\O_\l^\times, 
   \quad J := \O_{K,\Sigma}^\times.
$$
Lemma \ref{precft}(ii) shows that $\ker(\O_{K,\Sigma}^\times/(\O_{K,\Sigma}^\times)^p \to G/G^p)$ 
is generated by $\Delta$. 
Now we deduce from \eqref{ck} and Lemma \ref{elem}(i) that the image of 
$$
\Xset(K) \to \Hom(\prod_{v \in \Sigma}K_v^\times \times \prod_{\l\in\cP_0\cup\cP_1}\O_\l^\times,\bmu_p)
$$
is 
$\Hom((\prod_{v \in \Sigma}K_v^\times \times \prod_{\l\in\cP_0\cup\cP_1}\O_\l^\times)/\ld\Delta\rd,\bmu_p)$.
Restricting to characters unramified at $\cP_0\cup\cP_1$ proves (iii), since 
$\ker(\sgn) = \Hom((\prod_{v \in \Sigma}K_v^\times)/\ld\Delta\rd,\bmu_p)$.
\end{proof}

\section{Parity disparity ($p = 2$)}
\label{pd2}

Fix $T$ and $\Sigma$ as in \S\ref{notation}, a global \metabolic structure $\bq$ on $T$ 
as in Definition \ref{ahsdef}, and twisting data as in Definition \ref{twistdata}.
In this section we let $p=2$ and we study how the parity of $\dim_{\F_2}\Sel(T,\chi)$ 
varies as $\chi$ varies.  The main result is Theorem \ref{sllem,p=2}.  When $T = E[2]$ 
with an elliptic curve $E/K$, Theorem \ref{sllem,p=2} specializes to Theorem \ref{thma} 
of the Introduction, and we make 
Theorem \ref{thma} more explicit in Proposition \ref{tables}, Corollary \ref{6.8}, 
and Example \ref{50a1}.

Suppose throughout this section that $p=2$ and that \eqref{h2a} and \eqref{h2b} 
are satisfied.  Let $\Delta \in \O_{K,\Sigma}^\times$ be as in Lemma \ref{4.3}.

If $\chi \in \Xset(K)$, let $r(\chi) := \dim_{\F_2}\Sel(T,\chi)$, 
where $\Sel(T,\chi)$ is given by Definition \ref{sstwist}, the Selmer group 
for the twist of $T$ by $\chi$.
If $T = E[2]$ with the natural twisting data, then $r(\chi) = \dim_{\F_2}\Sel_2(E^\chi/K)$ 
by Proposition \ref{4.9}.

Recall the function $\com{\chi}{\chi'} := \dim_{\F_2}\alpha_v(\chi)/(\alpha_v(\chi)\cap\alpha_v(\chi'))$ 
of Definition \ref{comp}.  

\begin{defn}
\label{omegadef}
For every $v \in \Sigma$, define a map (of sets) $\omega_v : \Xset(K_v) \to \{\pm1\}$ by 
$$
\omega_v(\chi_v) := (-1)^{\com{\one_v}{\chi_v}} \chi_v(\Delta).
$$
\end{defn}

\begin{prop}
\label{5.1,p=2}
Suppose $\chi \in \Xset(K)$.  Then 
$$
r(\chi) \equiv r(\one_K) \pmod{2} \iff \prod_{v \in \Sigma} \omega_v(\chi_v) = 1.
$$
\end{prop}

\begin{proof}
We will deduce this from Theorem \ref{kramer}.  Fix $\d \in \D$ such that $\chi \in \Xset(\d)$. 
If $\l \in \cP_0 \cup \cP_2$, then $\chi_\l(\Delta) = 1$ by Lemma \ref{5.6}(i,iii).
If $\l \in \cP_1$, then $\chi_\l(\Delta) = -1$ if $\l\mid\d$, and 
$\chi_\l(\Delta) = 1$ if $\l\nmid\d$, by Lemma \ref{5.6}(ii).  
Therefore
$$
\prod_{\l\notin\Sigma} \chi_\l(\Delta) = (-1)^{|\{\l: \text{$\l \in \cP_1$ and $\l \mid \d$}\}|} 
   = (-1)^{w(\d)}
$$
so by Theorem \ref{kramer},
$$
r(\chi) \equiv r(\one_K) \pmod{2} \iff \prod_{v \in \Sigma} (\omega_v(\chi_v)\chi_v(\Delta)) 
   \prod_{v\notin\Sigma} \chi_v(\Delta) = 1.   
$$
Global class field theory shows that $\prod_v \chi_v(\Delta) = 1$, and the proposition follows.
\end{proof}

\begin{defn}
\label{ckxdef}
Define a (set) function $\Xset(K) \to \Z_{>0}$ measuring the ``size'' of a character $\chi$ by
$$
\mag{\chi} := \max\{\N\l : \text{$\chi$ is ramified at $\l$}\}
$$
If $X > 0$, let $\Xset(K,X) \subset \Xset(K)$ be the subgroup
$$
\Xset(K,X) := \{\chi\in\Xset(K) : \mag{\chi} < X\}.
$$
\end{defn}

\begin{defn} 
\label{epdef}
For every $v \in \Sigma$ define
$$
\delta_v := \frac{1}{|\Xset(K_v)|}\sum_{\chi\in\Xset(K_v)}\omega_v(\chi), \quad\text{and}\quad
   \delta := (-1)^{r(\one_K)}\prod_{v\in\Sigma}\delta_v.
$$
\end{defn}

Note that $-1+2/|\Xset(K_v)| \le \delta_v \le 1$ for every $v$ (since $\omega_v(\one_v) = 1$) and 
$\delta \in [-1,1]$.

\begin{lem}
\label{sllemlem}
$$
\frac{|\{\gamma\in\Gamma_1 : \prod_{v\in\Sigma}\omega_v(\gamma_v) = 1\}|}{|\Gamma_1|} 
   = \frac{1+\prod_{v\in\Sigma}\delta_v}{2}.
$$
\end{lem}

\begin{proof}
Let $N = |\{\gamma\in\Gamma_1 : \prod_{v\in\Sigma}\omega_v(\gamma_v) = 1\}|$.  
Since $\Gamma_1 = \prod_{v\in\Sigma}\Xset(K_v)$, we have 
$$
N - (|\Gamma_1|-N) = \sum_{\gamma\in\Gamma_1} \prod_{v\in\Sigma}\omega_v(\gamma_v)
   = \prod_{v\in\Sigma}\biggl(\,\sum_{\gamma_v\in\Xset(K_v)}\omega_v(\gamma_v)\biggr)
$$
and dividing both sides by $|\Gamma_1| = \prod_{v\in\Sigma}|\Xset(K_v)|$ yields
$
2{N}/{|\Gamma_1|}-1 = \prod_{v\in\Sigma}\delta_v.
$
The lemma follows.
\end{proof}

\begin{thm}
\label{sllem,p=2}
For all sufficiently large $X$,
$$
\frac{|\{\chi \in \Xset(K,X) : \text{$\dim_{\F_2}\Sel(T,\chi)$ is even}\}|}{|\Xset(K,X)|} 
   = \frac{1+\delta}{2}.
$$
\end{thm}

\begin{proof}
Suppose $X$ is large enough so that the natural group homomorphism $\eta : \Xset(K,X) \to \Gamma_1$ 
is surjective (this holds for all sufficiently large $X$ by Proposition \ref{cft}(i)).
By Proposition \ref{5.1,p=2}, the parity of $r(\chi)$ depends only on $\eta(\chi)$.  
Since $\eta$ is a homomorphism, all of its fibers have the same size, so by Proposition \ref{5.1,p=2}
$$
\frac{|\{\chi \in \Xset(K,X) : r(\chi) \equiv r(\one_K) \pmod{2}\}|}{|\Xset(K,X)|}
  =  \frac{|\{\gamma\in\Gamma_1 : \prod_{v\in\Sigma}\omega_v(\gamma_v) = 1\}|}{|\Gamma_1|}. 
$$
Now the theorem follows from Lemma \ref{sllemlem}.
\end{proof}

\begin{rem}
As the proof shows, the equality of Theorem \ref{sllem,p=2} holds with $\Xset(K,X)$ replaced by 
any subset $\B \subset \Xset(K)$ having the property that there is a subgroup $A \subset \Xset(K)$ 
such that the natural map $A \to \Gamma_1$ is surjective and $A\B = \B$.
\end{rem}

For the rest of this section, we fix an elliptic curve $E/K$ and take $T = E[2]$ 
with the natural twisting data of Definition \ref{etwda}.  
In this setting Proposition \ref{4.9} shows that 
Theorem \ref{sllem,p=2} specializes to Theorem \ref{thma} of the Introduction.
Proposition \ref{tables} below computes the $\delta_v$ for $E[2]$, 
in all cases when $v \nmid 2$, and 
in certain cases when $v \mid 2$.  
We first need the following lemma.

\begin{lem}
\label{twdelta}
Suppose $v \in \Sigma$ and $\psi\in\Xset(K_v)$.  Let $\omega_v$ be as in 
Definition \ref{omegadef}, and let $\omega_v^\psi$ be the corresponding quantity 
defined with the elliptic curve $E^\psi$ over $K_v$ in place of $E$.  Then for every $\chi\in\Xset(K_v)$, we have
$$
\omega_v^\psi(\chi)\omega_v^\psi(\psi) = \omega_v(\chi\psi).
$$
\end{lem}

\begin{proof}
Let $\compsi{\one_v}{\chi}$ be as given by Definition \ref{comp} 
for $E^\psi$ in place of $E$.  Then by Lemma \ref{ttd}, 
\begin{multline*}
\omega_v^\psi(\chi)\omega_v^\psi(\psi) 
   = (-1)^{\compsi{\one_v}{\chi}}\chi(\Delta)(-1)^{\compsi{\one_v}{\psi}}\psi(\Delta) \\
   = (-1)^{\com{\one_v}{\psi}+\com{\one_v}{\chi\psi}+\com{\one_v}{\psi}}\chi\psi(\Delta) 
   = \omega_v(\chi\psi).
\end{multline*}
\end{proof}

\begin{prop}
\label{tables}
Suppose $E$ is an elliptic curve over $K$, and $T = E[2]$ with the natural twisting data.  
For every $v \in \Sigma$, 
let $m_v^\pm := |\{\gamma \in \Xset(K_v) : \omega_v(\gamma) = \pm1\}|$, 
and let $c_v := |K_v^\times/(K_v^\times)^2|$, so $c_v = 4$ if $v \nmid 2\infty$.  Then 
we have the following table, where if $v\nmid\infty$ then 
``type'' denotes the Kodaira type of the N\'eron model.
$$
\renewcommand{\arraystretch}{1.1}
\begin{array}{|c||c|c|c|c||c|}
\hline
\text{type of $v$} & \hskip6pt\mv{+}\hskip6pt & \hskip6pt\mv{-}\hskip6pt & \hskip6pt \delta_v\hskip6pt \\
\hline
\hline
\text{real} & 1 & 1 & 0 \\
\hline
\text{complex} & 1 & 0 & 1 \\
\hline
\text{split multiplicative} & 1 & c_v - 1 & 2/c_v - 1 \\
\hline
\text{type $\mathrm{I}_\nu$ or $\mathrm{I}_\nu^*$, $\nu > 0$, not split multiplicative} & c_v - 1 & 1 & 1-2/c_v \\
\hline
\text{good reduction or type $\mathrm{I}_0^*$, $v \nmid 2$} & 4 & 0 & 1 \\
\hline
\text{type $\mathrm{II}$, $\mathrm{IV}$, 
   $\mathrm{II}^*$, $\mathrm{IV}^*$, $\Delta \in (K_v^\times)^2$, $v \nmid 2$} & 4 & 0 & 1 \\
\hline
\text{type $\mathrm{II}$, $\mathrm{IV}$, 
   $\mathrm{II}^*$, $\mathrm{IV}^*$, $\Delta \notin (K_v^\times)^2$, $v \nmid 2$} & 2 & 2 & 0 \\
\hline
\text{type $\mathrm{III}$, $\mathrm{III}^*$, $-1 \in (K_v^\times)^2$, $v \nmid 2$} & 4 & 0 & 1 \\
\hline
\text{type $\mathrm{III}$, $\mathrm{III}^*$, $-1 \notin (K_v^\times)^2$, $v \nmid 2$} & 2 & 2 & 0 \\
\hline
\end{array}
$$
\end{prop}

\begin{proof}
Most of the entries in the table follow directly from calculations of Kramer \cite{kramer}, 
using Lemma \ref{prekramer}.
For every $v$, we have $\omega_v(\one_v) = 1$, and by definition $\delta_v = (m_v^+ - m_v^-)/c_v$.

\smallskip\noindent{\em Case 1: $v$ archimedean.}
If $v$ is complex then there is nothing to check.  Suppose $v$ is real, and 
let $\chi : K_v^\times \to \pm1$ be the sign character, the nontrivial element of $\Xset(K_v)$. 
By Lemma \ref{prekramer} and \cite[Proposition 6]{kramer}, we have
\begin{equation}
\label{ht}
(-1)^{\com{\one_v}{\chi}} = -\chi(\Delta)
\end{equation}
so $\omega_v(\chi) = -1$.  Thus $m_v^+ = m_v^- = 1$.

\smallskip\noindent{\em Case 2: $v$ split multiplicative.}
In this case Lemma \ref{prekramer} and \cite[Proposition 1]{kramer} show that if 
$\chi \in \Xset(K_v)$ is nontrivial, then \eqref{ht} holds, so $\omega_v(\chi) = -1$.  
Thus $m_v^+ = 1$ and $m_v^- = c_v-1$.

\smallskip\noindent{\em Case 3: $v$ type $\mathrm{I}_\nu$ or $\mathrm{I}_\nu^*$, $\nu > 0$, 
not split multiplicative.}
In this case there is a $\psi\in\Xset(K_v)$, $\psi \ne \one_v$, 
such that $E^\psi$ is split multiplicative (see for example \cite[\S1.12]{serre1972}).  
Case 2 showed that $\omega_v^\psi(\psi) = -1$, and so by Case 2 and Lemma \ref{twdelta} we have 
$m_v^+ = c_v-1$, $m_v^- = 1$.

\smallskip\noindent{\em Case 4: $v$ good reduction or type $\mathrm{I}_0^*$, $v \nmid 2$.}
If $E$ has good reduction at $v$, then 
Lemma \ref{prekramer} and \cite[Proposition 3]{kramer} show that 
$(-1)^{\com{\one_v}{\chi}} = \chi(\Delta)$ for every $\chi \in \Xset(K_v)$, so $\omega_v(\chi) = 1$.  
If $E$ has  reduction type $\mathrm{I}_0^*$, then $E$ has a quadratic twist with good reduction, 
so by Lemma \ref{twdelta} we again have $\omega_v(\chi) = 1$ for every $\chi$.  In either case 
$m_v^+ = c_v = 4$, $m_v^- = 0$. 

\smallskip\noindent{\em Case 5: $v$ type $\mathrm{II}$, $\mathrm{IV}$, 
   $\mathrm{II}^*$, or $\mathrm{IV}^*$, $v \nmid 2$.}
In this case the number of connected components of the N\'eron model is odd, 
and $E$ has additive reduction at $v$, so $E(K_v)$ is $2$-divisible.  Hence 
$\alpha_v(\chi)$ is zero for every $\chi$, so $\com{\one_v}{\chi} = 0$, so 
$\omega_v(\chi) = \chi(\Delta)$.  
Thus if $\Delta \in (K_v^\times)^2$, then $m_v^+ = c_v = 4$ and $m_v^- = 0$, 
and if $\Delta \notin (K_v^\times)^2$, then $m_v^+ = m_v^- = 2$.

\smallskip\noindent{\em Case 6: $v$ type $\mathrm{III}$ or $\mathrm{III}^*$, $v \nmid 2$.}
Suppose $\chi \in \Xset(K_v)$, $\chi \ne \one_v$, 
and let $F$ be the corresponding quadratic extension of $K_v$.
In this case Tate's algorithm \cite{tatealg} shows that 
$$
E(K_v)[2] \cong \Z/2\Z, \quad E(K_v) = E(K_v)[2] \times B, \quad E(F) = E(F)[2] \times B'
$$ 
with profinite abelian groups $B, B'$ of odd order, and $\ord_v(\Delta)$ is odd.  
If $F \ne K_v(\sqrt{\Delta})$, then $E(F)[2] = E(K_v)[2]$, so $\N_{F/K_v}E(F) = 2E(K_v) = B$, 
so $\com{\one_v}{\chi} = 1$ by Lemma \ref{prekramer}.  If $F = K_v(\sqrt{\Delta})$, 
then $E(F)[2] = E[2]$ and $\N_{F/K_v}E(F) = E(K_v)$, so $\com{\one_v}{\chi} = 0$ by Lemma \ref{prekramer}.

Fix $u \in \O_v^\times$, $u \notin (\O_v^\times)^2$.  We have 
$\Xset(K_v) = \{\one_v,\chi_\Delta, \chi_u, \chi_{\Delta u}\}$, where $\chi_a$ is the 
quadratic character corresponding to $K_v(\sqrt{a})$.
The discussion above showed that $\com{\one_v}{\one_v} = \com{\one_v}{\chi_\Delta} = 0$ and 
$\com{\one_v}{\chi_u} = \com{\one_v}{\chi_{\Delta u}} = 1$.
We have $\chi_u(\Delta) = -1$, since $K_v(\sqrt{u})/K_v$ is unramified and $\ord_v(\Delta)$ is odd.
With $F = K_v(\sqrt{\Delta})$ we have
$$
\chi_\Delta(\Delta) = 1 \iff \Delta \in \N_{F/K_v}F^\times 
   \iff -1 \in \N_{F/K_v}F^\times \iff -1 \in (K_v^\times)^2,
$$
and with $F = K_v(\sqrt{\Delta u})$ we have
\begin{multline*}
\chi_{\Delta u}(\Delta) = 1 \iff \Delta \in \N_{F/K_v}F^\times 
   \iff -u \in \N_{F/K_v}F^\times \\ \iff -u \in (K_v^\times)^2 \iff -1 \notin (K_v^\times)^2,
\end{multline*}
Combining these facts gives the entries in the last two rows of the table. 
\end{proof}

\begin{cor}
\label{6.8}
\begin{enumerate}
\item
If $K$ has a real embedding, then for all sufficiently large $X$ we have
$$
|\{\chi \in \Xset(K,X) : \text{$r(\chi)$ is odd}\}|
   = |\{\chi \in \Xset(K,X) : \text{$r(\chi)$ is even}\}| = \frac{|\Xset(K,X)|}{2}.
$$
\item
If $K$ has no real embeddings, $E/K$ is semistable, and $E$ has multiplicative  
reduction at all primes above $2$, then for all sufficiently large $X$
$$
\frac{|\{\chi \in \Xset(K,X) : \text{$r(\chi)$ is even}\}|}{|\Xset(K,X)|} 
  = \frac{1+\delta}{2} \notin \{0, \textstyle\frac{1}{2}, 1\}.
$$
\end{enumerate}
\end{cor}

\begin{proof}
If $K$ has a real place $v$, then Proposition \ref{tables} shows that $\delta_v = 0$, 
so (i) follows from Theorem \ref{sllem,p=2}.

Under the hypotheses of (ii), Proposition \ref{tables} shows that
\begin{itemize}
\item 
$\delta_v = 1$ for every archimedean place and every place of good reduction, 
\item 
$|\delta_v| = \frac{1}{2}$ for every place $v \nmid 2$ of bad reduction, since $c_v = 4$,
\item 
$|\delta_v| = 1-2^{-[K_v:\Qp]-1} \in [\frac{3}{4},1)$ if $v \mid 2$, since $c_v = [K_v:\Qp]+2$.
\end{itemize}  
Thus $\delta \notin \{0,\pm1\}$ so (ii) follows from Theorem \ref{sllem,p=2} as well.
\end{proof}

\begin{exa}
\label{50a1}
\footnote{The proof of Example \ref{50a1} in the published version of this paper 
was incorrect, because we applied \cite[Theorem 1.3]{dokdok} incorrectly.  
The statement and proof here have been corrected.
We thank Lilybelle Cowland Kellock for pointing out the error.}Let 
$E$ be the elliptic curve labelled {\rm 50B1} in \cite{cremona}:
$$
y^2+xy+y = x^3 + x^2 - 3x - 1
$$
and let $K$ be a finite abelian extension of $\Q$ containing $\sqrt{-2}$, unramified at $5$.  
Then for all sufficiently large $X$, 
$$
\frac{|\{\chi \in \Xset(K,X) : \text{$r(\chi)$ is even}\}|}{|\Xset(K,X)|} 
   = \frac{1}{2} + \frac{(-1)^{[K:\Q]/2}}{2}\prod_{v \mid 2}(1-2^{-[K_v:\Q_2]-1}).
$$
As $K$ varies, these values are dense in the interval $[0,1]$.
\end{exa}

\begin{proof}
The discriminant of $E$ is $-2^5\cdot5^2$, which is a square in $K$, 
so $\Delta = 1$ in $K^\times/(K^\times)^2$.  Over $\Q_2$, $E$ has split 
multiplicative reduction, so $E$ has split multiplicative reduction at every  
prime of $K$ above $2$.
Over $\Q_5$, $E$ has Kodaira type II, and since $K/\Q$ is unramified at $5$, 
$E$ has Kodaira type II at all primes of $K$ above $5$.  
Further, since $K$ is unramified at $5$ we have $E(K)[2] = 0$.

Let $w(E/K_v)$ denote the local root number of $E$ over $K_v$.  
We have $w(E/K_v) = 1$ if $v \nmid 2 \cdot 5 \cdot \infty$.  
By \cite[Theorem 2]{rohrlich} we have 
$w(E/K_v) = -1$ if $v \mid 2$ or $v \mid \infty$, and $w(E/K_v) = 1$ if $v \mid 5$. 
Thus the global root number $w(E/K)$ is given by
$$
w(E/K) = \prod_v w(E/K_v) = (-1)^{n_\infty + n_2}
$$
where $n_2$ (resp., $n_\infty = [K:\Q]/2$) is the number 
of places of $K$ above $2$ (resp., above $\infty$).
By \cite[Theorem 1.3]{dokdok} (the ``$2$-Selmer parity conjecture''), 
combined with the Cassels pairing (see for example \cite[Proposition 2.1]{MRvisibility}), 
it follows that $\dim_{\F_2}\Sel_2(E/K) \equiv n_\infty + n_2 \pmod{2}$.

The field $K$ has no real embeddings, and $E$ has good reduction at 
all primes of $K$ not dividing $10$.  Hence Proposition \ref{tables} 
shows that the $\delta$ of Definition \ref{epdef} is given by 
$\delta = (-1)^{n_\infty+n_2}\prod_{v \mid 2}(2/|K_v^\times/(K_v^\times)^2|-1)$.  
For each $v$ dividing $2$ we have 
$$
|K_v^\times/(K_v^\times)^2| = 2^{[K_v:\Q_2]+2} 
$$
so the desired formula follows from Theorem \ref{sllem,p=2}.

To prove the final assertion, choose positive integers $m, t$. 
Choose finite abelian extensions 
$L_1, L_2$ of $\Q$, unramified at $5$, such that $[L_1:\Q] = t$, $[L_2:\Q] = m$, and 
$2$ splits completely in $L_1$ and remains prime in $L_2$. 
Let $K := L_1 L_2(\sqrt{-2})$.  
Then $K$ contains $\sqrt{-2}$, $K$ is abelian over $\Q$ and unramified at $5$, $[K:\Q] = 2mt$, 
$n_2 = [L_1:\Q] = t$, and $[K_v:\Q_2] = [L_2:\Q] = 2m$ if $v \mid 2$.
In this case the quantity on the right-hand side of the formula above is
$$
\frac{1}{2} + \frac{(-1)^{mt}}{2}(1-2^{-2m-1})^t.
$$
As $m$ and $t$ vary, the sets 
$$
\{\log((1-2^{-2m-1})^t): \text{$mt$ even}\}, \quad \{\log((1-2^{-2m-1})^t): \text{$mt$ odd}\}
$$ 
are both dense in $\R_{\le 0}$.  
It follows from the continuity of the exponential function that the set $\{(-1)^{tm}(1-2^{-2m-1})^t\}$ 
is dense in $[-1,1]$.  This completes the proof.
\end{proof}

Note that the Selmer rank and Mordell-Weil rank of $E^\chi$ are related by
\begin{align*}
\rk(E^\chi(K)) &\ge r(\chi) - \dim_{\F_2}E(K)[2] \\
\intertext{and if the $2$-part of the Shafarevich-Tate group $\Sh(E^\chi/K)$ is finite, then}
\rk(E^\chi(K)) &\equiv r(\chi) - \dim_{\F_2}E(K)[2] \pmod{2}.
\end{align*}
Thus by Theorem \ref{sllem,p=2} we expect that $\rk(E^\chi(K))$ 
is odd (and therefore at least one) for exactly 
$(1-(-1)^{\dim_{\F_2}E(K)[2]}\delta)/2$ of the twists $E^\chi$.
This leads to the following generalization of Goldfeld's conjecture 
\cite[Conjecture B]{goldfeld},
which follows from Theorem \ref{sllem,p=2} if we assume 
\begin{itemize}
\item
the $2$-parts of the Shafarevich-Tate groups of twists of $E$ are all finite,
\item
twists of rank at least $2$ are rare enough that they do not affect the average rank.
\end{itemize}

\begin{conj}
\label{gcK}
The average rank of the quadratic twists of $E/K$ is given by
$$
\lim_{X \to \infty}\frac{\sum_{\chi\in\Xset(K,X)}\rk(E^\chi(K))}{|\Xset(K,X)|} 
   = \frac{1-(-1)^{\dim_{\F_2}E(K)[2]}\delta}{2}.
$$
\end{conj}

\begin{exa}
\rm
This example shows that the fraction of even ranks given by Theorem \ref{sllem,p=2} 
does depend on the way we have chosen to order the twists.  Let $\Xset(K,X)$ be as above, 
and consider also another natural ordering
$$
\Xsetclassic(K,X) := \{\chi_d : d \in \O_K: |\N d| < X \}
$$
where $\chi_d$ is the character of $K(\sqrt{d})/K$.
Let $E$ be the elliptic curve 38B1 in \cite{cremona}
$$
y^2 + xy + y = x^3 + x^2 + 1
$$
and $K = \Q(i)$.  Then $r(\one_K) = 0$, and 
$E$ has split multiplicative reduction at the primes $(1+i)$ and $(19)$, 
and good reduction everywhere else.  
We have $|\Xset(K_{1+i})| = 2^4$ and $|\Xset(K_{19})| = 2^2$, and according to Proposition \ref{tables} 
we have $\omega_v(\chi_v) = 1$ if and only if $\chi_v = \one_v$, for $v = (1+i)$ or $(19)$ and 
$\chi_v \in \Xset(K_v)$.

If $X > \N(19) = 19^2$, then the images 
of the characters $\chi$ in the group $\Xset(K,X)$ are uniformly distributed in 
$\Xset(K_{1+i}) \times \Xset(K_{19})$.  Hence under the map
$$
\Xset(K,X) \too \Xset(K_{1+i}) \times \Xset(K_{19}) \map{\omega_{1+i} \times \omega_{19}} \{\pm1\}\times\{\pm1\}
$$
exactly $\frac{1}{16} \cdot \frac{1}{4} = \frac{1}{64}$ of them map to $(1,1)$ and 
$\frac{15}{16} \cdot \frac{3}{4} = \frac{45}{64}$ of them map to $(-1,-1)$.  Hence 
by Proposition \ref{5.1,p=2}, $r(\chi)$ is even 
for exactly $\frac{23}{32} = \frac{1}{2} + \frac{7}{32}$ of the $\chi\in\Xset(K,X)$.  
This is the content of Theorem \ref{sllem,p=2} in this case.

Now consider the density using $\Xsetclassic(K,X)$ instead of $\Xset(K,X)$.  The quadratic characters of 
$K$ correspond bijectively to squarefree integers $d \in \Z[i]$ modulo $\pm1$, and $\Xsetclassic(K,X)$ 
corresponds to $d$ with $\N d < X$.  These characters no longer map uniformly to 
$\Xset(K_{1+i}) \times \Xset(K_{19})$; for example, the fraction of characters unramified at $(19)$ 
(i.e., the fraction of squarefree $d$'s that are not divisible by $19$) is 
$19^2/(19^2+1)$, not $1/2$.  Of those that are unramified, half of the $d$'s are squares modulo $19$.
Reasoning in this way we see that under the map 
$$
\Xsetclassic(K,X) \too \Xset(K_{1+i}) \times \Xset(K_{19}) \map{\omega_{1+i} \times \omega_{19}} \{\pm1\}\times\{\pm1\}
$$
the fraction mapping to $(1,1)$ is  
$$
\biggl(\frac{1}{8}\cdot\frac{2}{2+1}\biggr)\cdot\biggl(\frac{1}{2}\cdot\frac{19^2}{19^2+1}\biggr) 
   = \frac{1}{12}\cdot\frac{361}{724} = \frac{361}{8688},
$$
and the fraction mapping to $(-1,-1)$ is $\frac{11}{12} \cdot \frac{363}{724} = \frac{1331}{2896}$.
We conclude by Proposition \ref{5.1,p=2} that 
$$
\lim_{X \to \infty} \frac{|\{\chi\in\Xsetclassic(K,X) : \text{$r(\chi)$ is even}\}|}{|\Xsetclassic(K,X)|}
   = \frac{361}{8688} + \frac{1331}{2896} = \frac{2177}{4344} = \frac{1}{2} + \frac{5}{4344}.
$$
\end{exa}

\section{Parity ($p > 2$)}
\label{pdp}

In this section suppose that $p>2$ and that \eqref{h2a}, \eqref{h2b} are satisfied.  
We will study how the parity of $\dim_{\F_p}\Sel(T,\chi)$ 
varies as $\chi$ varies. 

Recall that $\Xset(K) = \coprod_{\d\in\D}\Xset(\d)$.  
If $\chi \in \Xset(\d)$, let 
$$
w(\chi) := w(\d), \qquad r(\chi) := \dim_{\Fp}\Sel(T,\chi),
$$ 
where $\Sel(T,\chi)$ is given by Definition \ref{sstwist}, the Selmer group 
for the twist of $T$ by $\chi$.
Similarly, if $\gamma\in\Gamma_\d$ we let $r(\gamma) := \dim_{\Fp}\Sel(T,\gamma)$.

Let $\eta : \Xset(K) \to \Gamma_1$ be the natural homomorphism.

\begin{defn}
\label{epmdef}
Define 
$$
\parity := \frac{|\{\gamma\in\Gamma_1 : \text{$r(\gamma)$ is odd}\}|}{|\Gamma_1|}.
$$
\end{defn}

Note that $\rho$ cannot be $1/2$, since $|\Gamma_1|$ is odd.
The main result of this section is the following.

\begin{thm}
\label{sllem}
\begin{enumerate}
\item
If $p \nmid [K(T):K]$, then for all sufficiently large $X$
$$
\frac{|\{\chi \in \Xset(K,X) : \text{$r(\chi)$ is odd}\}|}{|\Xset(K,X)|} = \rho.
$$
\item
If $p \mid [K(T):K]$, then
$$
\lim_{X\to\infty}\frac{|\{\chi \in \Xset(K,X) : \text{$r(\chi)$ is odd}\}|}{|\Xset(K,X)|}
   = \frac{1}{2}.
$$
\end{enumerate}
\end{thm}

\begin{proof}[Proof of Theorem \ref{sllem}(i)]
By Corollary \ref{kramercor}, $r(\chi) \equiv r(\eta(\chi)) + w(\chi) \pmod{2}$ 
for every $\chi$.  Since $p \nmid [K(T):K]$, Lemma \ref{4.2}(ii) shows that 
$w(\chi)$ is even for every $\chi$, 
so $r(\chi)$ is odd if and only if $r(\eta(\chi))$ is odd.  
By Proposition \ref{cft}(i), for all $X$ sufficiently large,
$\eta$ restricts to a surjective homomorphism of finite groups $\Xset(K,X) \to \Gamma_1$.  
In particular all fibers have the same size, so for large $X$
$$
|\{\chi \in \Xset(K,X) : \text{$r(\chi)$ is odd}\}| = 
   |\{\gamma\in\Gamma_1 : \text{$r(\gamma)$ is odd}\}|
   \;\frac{\Xset(K,X)}{|\Gamma_1|}
$$
which proves assertion (i) of the theorem.
\end{proof}

The rest of this section is devoted to the proof of Theorem \ref{sllem}(ii).  
Order the primes of $K$ not in $\Sigma$ by norm, $\N\l_1 \le \N\l_2 \le \cdots$.
For every $n$, let $C_n \subset \Xset(K)$ be the subgroup
$$
C_n := \{\chi \in \Xset(K) : \text{$\chi$ is unramified outside of $\Sigma \cup \{\l_1,\ldots,\l_n\}$}\}.
$$
For every $\gamma \in \Gamma_1$ define
$$
s_n(\gamma) := \frac{|\{\chi\in C_n : \text{$\eta(\chi) = \gamma$ and $w(\chi)$ is even}\}|}
     {|\{\chi\in C_n : \eta(\chi) = \gamma\}|} - \frac{1}{2}.
$$
We will show that $\lim_{n \to \infty} s_n(\gamma) = 0$ for every $\gamma \in \Gamma_1$.

\begin{lem}
\label{l0}
\begin{enumerate}
\item
If $\bmu_p \not\subset K_{\l_n}^\times$, then $C_n = C_{n-1}$.
\item
If $\bmu_p \subset K_{\l_n}^\times$, then there is a $\psi\in C_n$, 
ramified at $\l_{n}$ and unramified at $\l_1,\ldots,\l_{n-1}$, such that 
$C_n = \coprod_{i=0}^{p-1}\psi^i C_{n-1}$.  If $\chi\in C_{n-1}$ then
$$
w(\psi^i\chi) = 
   \begin{cases}
   w(\chi) & \text{if $p \mid i$} \\
   w(\chi) + k & \text{if $p \nmid i$ and $\l_n \in \cP_k$.}
   \end{cases} 
$$
\end{enumerate}
\end{lem}

\begin{proof}
If $\bmu_p \not\subset K_{\l_{n}}^\times$ then by local class field theory 
no character of order $p$ can ramify at $\l_{n}$, so $C_{n} = C_{n-1}$.
If $\bmu_p \subset K_{\l_{n}}^\times$, then by Proposition \ref{cft}(ii) there is a character $\psi \in C_n$ 
ramified at $\l_{n}$ and unramified at $\l_1,\ldots,\l_{n-1}$.  
The restriction of $\psi$ generates $\Hom(\O_{\l_n}^\times,\bmu_p)$, so $\psi$ generates $C_n/C_{n-1}$.  
If $\l_n \in \cP_k$ and $p \nmid i$ then 
$w(\psi^i\chi) = w(\psi^i) + w(\chi) = k + w(\chi)$.
This proves the lemma.
\end{proof}

Let $\Delta \in \O_{K,\Sigma}^\times$ be as in Lemma \ref{4.3}, and let 
$\sgn : \Gamma_1 \to \bmu_p$ be the homomorphism $\gamma \mapsto \prod_{v \in \Sigma} \gamma_v(\Delta)$ 
of Definition \ref{gpmdef}.

\begin{lem}
\label{l2}
There is an $N \in \Z_{>0}$ such that if $n \ge N$ then $s_n(\gamma)$ depends only on $n$ and 
$\sgn(\gamma)$.
\end{lem}

\begin{proof}
By Proposition \ref{cft}(iii), if $n$ is large enough then for every $\gamma\in \Gamma_1$ 
with $\sgn(\gamma) = 1$, there is 
a character $\psi_\gamma \in C_n$ ramified only at primes in $\Sigma\cup\cP_2$, such that $\eta(\psi_\gamma) = \gamma$.  

Suppose $\gamma_1, \gamma_2 \in \Gamma_1$ and $\sgn(\gamma_1) = \sgn(\gamma_2)$.  
Let $\gamma := \gamma_1^{-1}\gamma_2$.  Then multiplication by $\psi_{\gamma}$ gives a bijection
$$
\{\chi\in C_n : \eta(\chi) = \gamma_1\} \too \{\chi\in C_n : \eta(\chi) = \gamma_2\}.
$$
Further, since $\psi_\gamma$ is unramified at primes in $\cP_1$ we have 
$w(\chi) \equiv w(\psi_\gamma\chi) \pmod{2}$ for every $\chi \in C_n$.  
Thus $s_n(\gamma_1) = s_n(\gamma_2)$, which proves the lemma.
\end{proof}

Define $S_n = \frac{1}{|\Gamma_1|}\sum_{\gamma\in\Gamma_1}s_n(\gamma)$, the average of the $s_n(\gamma)$.

\begin{lem}
\label{l1}
Suppose $n > N$ with $N$ as in Lemma \ref{l2}.  
If $\bmu_p \not\subset K_{\l_{n}}^\times$ let $\psi := \one_K \in C_n$, and  
if $\bmu_p \subset K_{\l_{n}}^\times$ let $\psi \in C_n$ be as in Lemma \ref{l0}(ii). 
In either case let $\e := (-1)^k$ where $\l_n \in \cP_k$, and 
$\bar\psi := \eta(\psi) \in \Gamma_1$.
Then 
$$
s_n(\gamma) = 
\begin{cases}
\frac{1+\e(p-1)}{p}s_{n-1}(\gamma) & \text{if $\sgn(\bar\psi) = 1$,} \\
\frac{1-\e}{p}s_{n-1}(\gamma) + \e S_{n-1} & \text{if $\sgn(\bar\psi) \ne 1$.}
\end{cases}
$$
\end{lem}

\begin{proof}
If $\bmu_p \not\subset K_{\l_{n}}^\times$, then $\bar\psi = \one$,  
$C_{n} = C_{n-1}$ by Lemma \ref{l0}(i), and $\l_n \in \cP_0$ by definition,
so $\e = 1$ and $s_n(\gamma) = s_{n-1}(\gamma)$ for every $\gamma$.  Thus 
the formula of the lemma holds in this case.

Suppose now that $\bmu_p \subset K_{\l_{n}}^\times$, so $\psi$ is ramified at $\l_n$.  
Then for every $\gamma\in\Gamma_1$, $\chi\in C_{n-1}$, 
and $0 \le i < p$, Lemma \ref{l0}(ii) shows that
$$
\text{$\eta(\psi^i\chi) = \gamma$ and $w(\psi^i\chi)$ is even} \iff 
   \begin{cases}\text{$i=0$, $\eta(\chi) = \gamma$, and $w(\chi)$ is even, or} \\
   \text{$i\ne 0$, $\eta(\chi) = \bar\psi^{-i}\gamma$, and $(-1)^{w(\chi)}=\e$.}\end{cases}
$$
Thus, using that $C_n = \coprod_{i=0}^{p-1}\psi^i C_{n-1}$ by Lemma \ref{l0}(ii), we have
$$
1/2+s_n(\gamma) = \frac{(1/2+s_{n-1}(\gamma)) + \sum_{i=1}^{p-1}(1/2 + \e s_{n-1}(\bar\psi^i\gamma))}{p}
$$
or equivalently
\begin{equation}
\label{tag2}
s_n(\gamma) = \frac{(1-\e)s_{n-1}(\gamma) + \e\sum_{i=0}^{p-1}s_{n-1}(\bar\psi^i\gamma)}{p}.
\end{equation}
If $\sgn(\bar\psi) = 1$ then $s_{n}(\bar\psi^i\gamma) = s_{n-1}(\gamma)$ 
for every $i$ by Lemma \ref{l2}, and \eqref{tag2} becomes $s_{n}(\gamma) = \frac{1+(p-1)\e}{p}s_{n-1}(\gamma)$.
If $\sgn(\bar\psi) \ne 1$ then $\sgn(\bar\psi)$ generates $\bmu_p$, so (using Lemma \ref{l2})
$$
\sum_{i=0}^{p-1}s_{n-1}(\bar\psi^i\gamma) = \frac{p}{|\Gamma_1|}\sum_{\varphi\in\Gamma_1}s_{n-1}(\varphi)
   = p S_{n-1}.
$$ 
Now the final equality of the lemma follows from \eqref{tag2}.
\end{proof}

\begin{cor}
\label{l4}
If $p \mid[K(T):K]$ and $\gamma\in\Gamma_1$, then 
$
\lim_{n\to\infty} s_n(\gamma) = 0.
$ 
\end{cor}

\begin{proof}
Averaging over $\gamma\in\Gamma_1$ in Lemma \ref{l1} shows that
$$
S_n = 
\begin{cases}
S_{n-1} & \text{if $\e = 1$} \\
\frac{2-p}{p}S_{n-1} & \text{if $\e = -1$}.
\end{cases}
$$
If $p \mid [K(T):K]$, then $\cP_1$ is infinite by Lemma \ref{4.2}, and $\e= -1$ 
whenever $\l_n \in \cP_1$, so we deduce that $\lim_{n\to\infty} S_n = 0$.
Applying Lemma \ref{l1} again proves the corollary.
\end{proof}

\begin{proof}[Proof of Theorem \ref{sllem}(ii)]
By Corollary \ref{kramercor},  $r(\chi) \equiv r(\eta(\chi)) + w(\chi) \pmod{2}$ 
for every $\chi$, so 
\begin{equation}
\label{cie}
\text{$r(\chi)$ is odd $\iff w(\chi) \not\equiv r(\eta(\chi)) \pmod{2}$.}
\end{equation}

We may assume (Proposition \ref{cft}(i)) that $n$ is large enough so that 
the map $\eta : C_n \to \Gamma_1$ is surjective.  
Using \eqref{cie}, for every $\gamma\in\Gamma_1$ we have
$$
\frac{|\{\chi \in C_n : \text{$\eta(\chi) = \gamma$ and $r(\chi)$ is odd}\}|}
   {|\{\chi \in C_n : \text{$\eta(\chi) = \gamma$}\}|}
   = 1/2 - (-1)^{r(\gamma)} s_n(\gamma).
$$
Since $p\mid[K(T):K]$, Corollary \ref{l4} shows that $\lim_{n\to\infty}s_n(\gamma) = 0$, so
$$
\lim_{n \to \infty}\frac{|\{\chi \in C_n : \text{$\eta(\chi) = \gamma$ and $r(\chi)$ is odd}\}|}
   {|\{\chi \in C_n : \text{$\eta(\chi) = \gamma$}\}|}
   = 1/2.
$$
This holds for every $\gamma \in \Gamma_1$, so
$$
\lim_{X \to \infty}\frac{|\{\chi \in \Xset(K,X) : \text{$r(\chi)$ is odd}\}|}{|\Xset(K,X)|}
   =\lim_{n \to \infty}\frac{|\{\chi \in C_n : \text{$r(\chi)$ is odd}\}|}{|C_n|}
   = 1/2.
$$
This completes the proof of Theorem \ref{sllem}
\end{proof}

\begin{rem}
Fix an elliptic curve $E/K$, and let $p$ vary.  If $E$ does not have complex multiplication, 
then Serre's theorem \cite{serre1972} shows that 
$p\mid[K(T):K]$ for all but finitely many $p$, so Theorem \ref{sllem}(ii) shows that 
for all but finitely many $p$, half of the twists by characters of order $p$ 
have even $p$-Selmer rank and half have odd $p$-Selmer rank.
\end{rem}

\end{document}